\documentclass{article}
\usepackage[utf8]{inputenc}
\usepackage{amsthm}
\usepackage{amsmath}
\usepackage{amssymb}
\usepackage{amsfonts}
\usepackage{mathrsfs}
\usepackage{array}
\usepackage[usenames]{color}
\usepackage{wrapfig}
\usepackage[margin=1in]{geometry}
\usepackage{graphicx}
\usepackage{float}
\usepackage{epstopdf}
\usepackage{mathtools}
\usepackage[square,numbers,sort&compress]{natbib}
\usepackage{xcolor}
\usepackage{tikz-cd}
\usepackage{xypic}
\usepackage{multirow}
\usepackage{adjustbox}
\usepackage{subfig}
\usepackage{bbm}
\usepackage[ruled]{algorithm2e}
\usepackage{hyperref}
\usepackage{quiver}
\usepackage{authblk}

\newtheorem{theorem}{Theorem}
\newtheorem{proposition}{Proposition}
\newtheorem{definition}{Definition}

\newtheorem{lemma}[proposition]{Lemma}
\newtheorem{example}{Example}
\newtheorem{corollary}[proposition]{Corollary}

\def\shf{\mathcal}     
\def\cshf{\mathfrak}   
\def\setsys{\mathfrak} 
\def\id{\textrm{id}}   
\newcommand{\pwr}{{\mathcal{P}}}  
\def\cat{\mathbf}      
\DeclareMathOperator{\vspan}{span} 
\newcommand{\blockpar}{\parindent 0pt \parskip 5pt}

\blockpar

\newcommand{\caj}[1]{{\color{red}{CAJ: #1}}}

\newcommand{\gal}[2]{#1 \hspace{-.1em} \mathop{:} \hspace{-.1em} #2}

\newcommand{\rawson}[1]{{\color{purple}{Michael Rawson: #1}}}

\newcommand{\tup}[1]{\left< #1 \right>}			
\newcommand{\func}[3]{#1 \colon #2 \rightarrow #3}	
\newcommand{\fig}[1]{Fig.~\ref{#1}}
\newcommand{\equ}[1]{(\ref{#1})}
\renewcommand{\sec}[1]{Sec.\ \ref{#1}}
\newcommand{\rem}[1]{}
\newcommand{\h}[1]{\mbox{#1}}
\newcommand{\cl}{\setsys{B}}

\newcommand{\st}{ \mathrel{\colon} }	

\newcommand{\sub}{\subseteq}		

\renewcommand{\int}{\cap}

\renewcommand{\(}[1]{\begin{equation} \label{#1}}
\renewcommand{\)}{\end{equation}}

\newcommand{\halffunc}[1]{ 
	\left\{ \begin{array}{ll} #1 \end{array} \right. 
}

\newcommand{\mypng}[3]{
	\begin{figure}[htbp]
	\begin{center}
	\includegraphics[scale=#1]{#2.png}
	\caption{#3}
	\label{#2}
	\end{center}
	\end{figure}
}

\blockpar

\title{The Topological Structures of the Orders of Hypergraphs}

\author[1,2]{Robert E.\ Green}
\author[1,3,*]{Cliff A.\ Joslyn}
\author[1]{Audun Myers} 
\author[4]{Michael G.\ Rawson}
\author[5]{Michael Robinson}
\affil[1]{Pacific Northwest National Laboratory, Complex Data Models}
\affil[2]{University of Pennsylvania, Mathematics}
\affil[3]{Binghamton University (SUNY), Systems Science and Industrial Engineering}
\affil[4]{Microsoft Corp.}
\affil[5]{American University, Math and Statistics}
\affil[*]{Corresponding author: cliff.joslyn@pnnl.gov, cajoslyn@binghamton.edu}

\date{April 2025}

\begin{document}

\maketitle


\begin{abstract}

We provide first a categorical exploration of, and then completion of the mapping of the relationships among, three fundamental perspectives on binary relations: as the incidence matrices of hypergraphs, as the formal contexts of concept lattices, and as specifying topological cosheaves of simplicial (Dowker) complexes on simplicial (Dowker) complexes. We provide an integrative, functorial framework combining previously known with  three new results: 1) given a binary relation, there are order isomorphisms among the bounded edge order of the intersection complexes of its dual hypergraphs and its concept lattice; 2) the concept lattice of a context is an isomorphism invariant of the Dowker cosheaf (of abstract simplicial complexes) of that context; and 3) a novel Dowker cosheaf (of chain complexes) of a relation  is an isomorphism invariant of the concept lattice of the context that generalizes Dowker's original homological result. We illustrate these concepts throughout with a running example, and demonstrate relationships to past results.

\end{abstract}

\tableofcontents


\section{Introduction}

In this paper we explore and demonstrate some novel and compelling functorial mappings among and interpretations of a collection of discrete mathematical structures. Stemming from the most simple and fundamental object of a binary relation, some critical relationships to high order graph and network structures, topological complexes, and order- and lattice-theoretical are well known. But we have discovered that that story is far from complete. 


To begin, some of the mathematical structures most significant for relational data science include graphs (taken as networks) and set systems (sets, or indexed families, or subsets, sometimes taken as lists of lists of items).
And recent years have seen an increasing emphasis on the value of generalizing graphs and networks to their high-order forms as  mathematical hypergraphs. In turn, hypergraphs provide not only a generalization of the connectivity of graphs from pairwise  to multi-way relations (every graph $G$ being a 2-uniform hypergraph), but also operate elegantly in a more general relational context, as follows. 

First, consider a (finite) undirected multi-graph $G = \tup{V,E}$, together with its incidence matrix $I(G)$, as a relational structure mapping rows as vertices $v \in V$ to columns as their incident undirected edges $e \in E$, each such edge $e \in {V \choose 2}$ being an unordered pair of vertices. Then $I(G)$ is highly constrained to have exactly two elements in each column, corresponding to the two vertices in each graph edge. But the incidence matrix $I(H)$ of a (finite) multi-hypergraph $H = \tup{V,E}$, where now for all $e \in E, e \in 2^V$, is an arbitrary (finite) Boolean matrix. As such, it is also the Boolean characteristic matrix of an arbitrary finite binary relation $R \sub V \times E$. In this way, hypergraphs provide a compelling interpretation of the most foundational discrete math structure, a general binary relation.

But the next part of the story is that hypergraphs $H$ also have natural {\em topological} interpretations, usually expressed in terms of abstract simplicial complexes (ASCs, also called Dowker complexes). ASCs are constructed from hypergraphs (below we will denote these as $DH(H)$) by closing them by subset inclusion, and their simplicial homologies have proven valuable for multiple purposes \cite{BeAAbR18}. Although this is the most natural way to turn a hypergraph into an ASC, there are a wide variety of approaches which have yielded useful homological invariants  \cite{purvine_Hgraph_homology}. 

Yet another connection between hypergraphs and topological spaces can be seen by considering the hyperedges themselves as a subbasis for a finite Alexandrov topology $T(H)$ on the vertices. Then open sets of $T(H)$ can be constructed by calculating all the unions of all the intersections of the hyperedges (everything here finite). This construction is natural on objects, but in a categorical context requires careful consideration of morphisms, making this issue a ripe target  for future work. 




But after admiring these various interpretations of objects determined by finite binary relations in terms of hypernetworks and finite topologies and complexes, we'd note that a {\em third} perspective on them has not been as heavily researched, and that is their interpretation in the context of ordered structures and lattices. Of course, every set system $E \sub 2^V$ is representable as a sublattice of the Boolean lattice $\tup{ 2^V, \sub }$ of the power set $2^V$, where hyperedges $e \in E$ are simply ordered by inclusion.
Also order complexes are important topological complexes that are built from the chains in this lattice \cite{priestley_1970}.
    It is also well known that the finite Alexandrov topologies $T(H)$ of hypergraphs (and thus of their incidence matrices $I(H)$) are bijective with both finite orders \cite{Barmak2011} and finite distributive lattices \cite{birkhoff}. 


But in this paper we are interested in a different, and, we think, more significant order theoretical interpretation and results for hypergraphs. Consider two hyperedges $e,e' \in E$ of a multi-hypergraph $H=\tup{V,E}$. Considering them within the edge order $\tup{E,\sub}$ can clearly only reveal information about their inclusion, whether $e \sub e', e' \sub e$, or they are non-comparable. But there are {\em two} ways that a pair of hyperedges can be noncomparable: they could be disjoint, with $e \cap e' = \emptyset$; or they could be ``properly intersecting'' with $e \cap e' \neq \emptyset$, but also both $e \setminus e' \neq \emptyset$ and $e' \setminus e \neq \emptyset$. 

In other words, to recover complete information about any pair of edges, in addition to noting their comparability by inclusion, we need to also calculate their intersection $e \cap e'$. It follows that to recover complete information about the entire hypergraph, we need to close it not by subset to form the Dowker complex $DH(H)$, and not by intersection {\em and} union to form the Alexandrov topology $T(H)$, but by intersections $e \cap e'$ {\em alone} to recover what we call the {\bf intersection complex} $H^\cap$. 

It turns out that this intersection complex object $H^\cap$ is well understood: it is a {\bf concept lattice} $\mathfrak{B}(H)$, which have been studied for quite a while in the sub-field of {\bf formal concept analysis (FCA)}. Concept lattices have been used in data science for decades in problems ranging from categorization to visualization. In recording the full intersection structure of a hypergraph, they also capture the Galois correspondence inherent in the  binary relation underlying the incidence matrix $I(H)$ of the hypergraph $H$, which is itself described as the lattice's {\bf formal context}.%
    \footnote{We note with regret some historical issues with notation: whereas we'd notate a multihypergraph $H=\tup{V,E,I}$, the tradition in FCA is to describe a formal context $C=\tup{G,M,I}$ for the identical structure. We will use each as appropriate, hopefully without confusion.}
    As will be detailed below, this Galois correspondence on the incidence matrix $I(H)$ relates certain pairs $K = \tup{A,B}$, called {\bf formal concepts}, of a set $A \sub V$ of vertices (rows) and a set $B \sub E$ of hyperedges (columns) in an adjoint order relation:  two such  pairs $K_1=\tup{A_1,B_1}$ and $K_2=\tup{A_2,B_2}$ are  ordered $K_1 \le K_2$ when $A_1 \sub A_2$ if and only if $B_1 \supseteq B_2$. In this way, we can make a concrete interpretation of hypergraphs as concept lattices, and {\it vice versa}. 

Where the edge order $\tup{E,\sub}$ relates hyperedges of a hypergraph $H$ order theoretically, the concept lattice $\mathfrak{B}(H)$ relates their intersections $e \cap e'$ in the intersection complex $H^\cap$ in terms of what we thereby call {\bf Galois pairs}, that is, pairs of sets we denote $\gal{A}{B}$ with $A \sub V, B \sub E$. The formal concepts $\tup{A,B}$ are thereby Galois pairs, but remarkably, such Galois structures also show up in other structures related to the hypergraph $H$, in particular its Dowker complex $DH(H)$. A functorial interpretation of the topological properties of Dowker complexes reveals their Galois structures via their sheaf representations. In particular, the cosheaf of a Dowker complex, which we thereby call a {\bf Dowker cosheaf}, codes the Galois relation of its corresponding concept lattice.



Cosheaves ascribe (usually algebraic) data to the open sets of a topological space, and thus we must ask about their origins from a simple hypergraph $H$. The answer is that
both the topology and the data arise as subsets, the open sets of the topology from sets of vertices $V$, and the data from sets of hyperedges $E$.  In other words, the Dowker complex $DH(H)$ builds a topology on $V$ while using the transpose of the incidence matrix $I$ builds a topology on $E$.  The Galois correspondence is then a correspondence between these two topologies, Galois pairs $\gal{A}{B}$ of these distinct open sets.  Fortunately, this correspondence is functorial, depending crucially upon the difference between formal contexts and relations, which is to say that any transformation between incidence matrices $I(H)$ that respects the presence of incidences yields continuous maps where appropriate.  

Adhering to a categorical perspective throughout our consideration of these varied perspectives of the same data also makes clear the importance of the morphisms in each category. The categories for binary relations, formal contexts, and multi-hypergraphs all have the same objects but different requirements on what it means to be a morphism. Morphisms, after all, are meant to preserve structure. The structures traditionally emphasized when thinking about these objects from different perspectives are different. For relations, a morphism requires that if two elements are related in the domain then they are also related in the codomain, for multi-hypergraphs a morphism requires hyperedges to be sent to hyperedges (just like edges to edges in graphs), and for contexts a morphism requires that concepts are sent to concepts.


We conclude by explicating the functorial connection between an abstract simplicial complex and its simplicial chain complex, which lifts chain complexes neatly into the costalks of the Dowker cosheaf.  This yields a new algebraic invariant for a hypergraph $H$ simultaneously for its concept lattice $\mathfrak{B}(H)$, its Dowker complex $DH(H)$, and its topology $T(H)$. 



This paper identifies as original contributions three new and related correspondences between concepts, concept lattices, the Dowker construction, hypergraphs, and cosheaves:

\begin{enumerate}

\item (Theorem \ref{thm:concept_iso_intclose} in Section \ref{sec:hypergraph_complexes}) For a context $C=(G,M,I)$,
  there are order isomorphisms among both of the bounded lattices of the intersection complexes of its dual hypergraphs (taking $C$ as an incidence matrix) and its concept lattice.

\item (Theorem \ref{thm:concept_dowker} in Section \ref{sec:concept_cosheaves}) The concept lattice of a context is an isomorphism invariant of the Dowker cosheaf (of abstract simplicial complexes) of that context.
  That is, if Dowker cosheaves for two contexts are isomorphic, then they correspond to isomorphic concept lattices.  Moreover, if the concepts in the concept lattice are labeled with Galois notation, then these data can recover the Dowker cosheaf. 
  
\item (Theorem \ref{thm:concept_iso_cosheaf_of} in Section \ref{sec:dowker_homology}) We define a novel Dowker cosheaf (of chain complexes) of a context, and show that it is an isomorphism invariant of the concept lattice of the context that generalizes Dowker's original homological result.
  
\end{enumerate}


\rem{

\begin{figure}
  \begin{center}
\includegraphics[width = 4in]{figures/overview}
\caption{A summary of the relationships discussed in this paper.}
\label{fig:overview}
  \end{center}
\end{figure}

}

Figure \ref{figures/categories} shows a functorial diagram representing the relationships between different categories discussed in the paper, which are introduced in \sec{categories} below. And \fig{figures/objects3} shows relationships between the different mathematical objects.




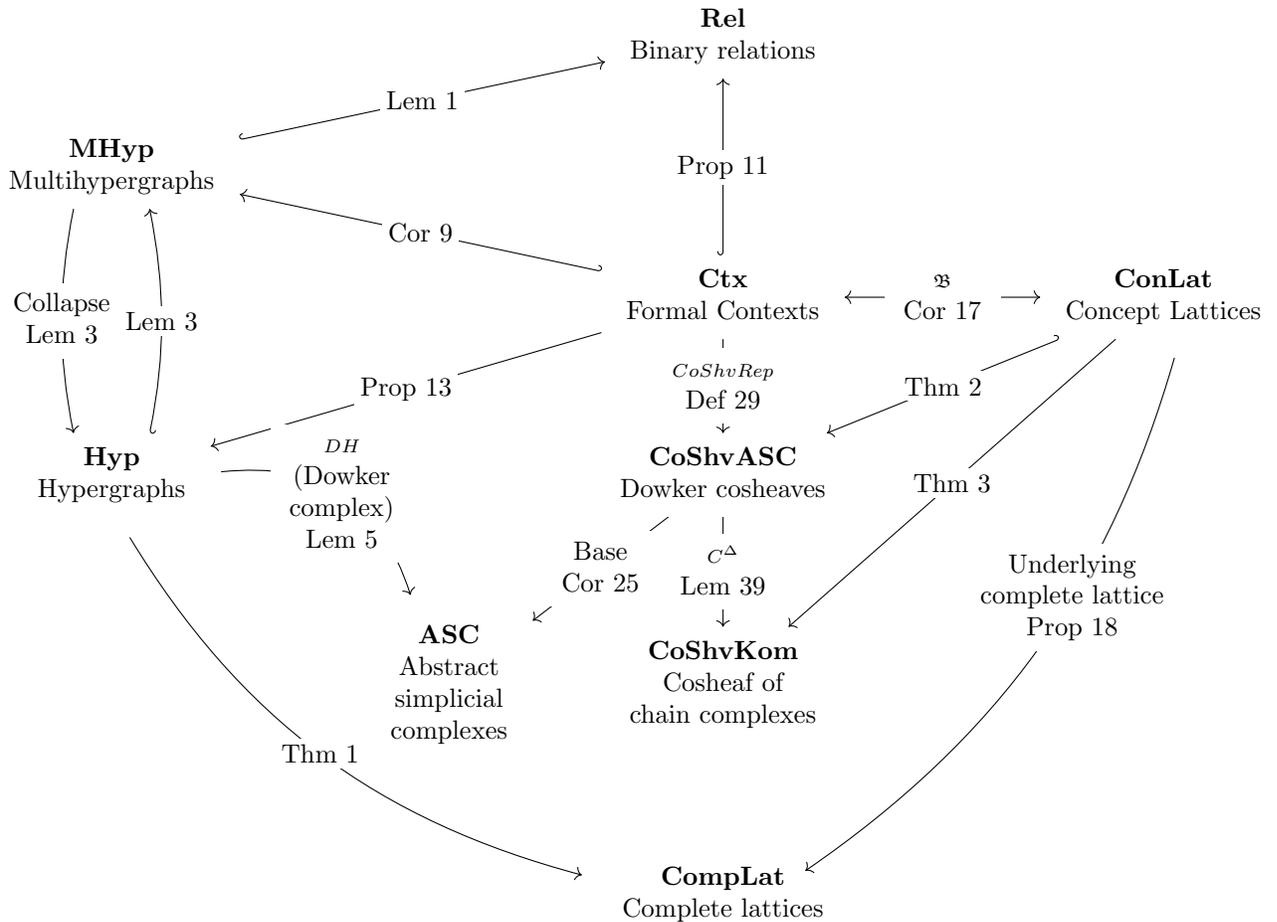
\begin{figure}
    \centering

\[\begin{tikzcd}
	&&& \begin{array}{c} \mathbf{Rel} \\ \hbox{Binary relations} \end{array} \\
	\begin{array}{c} \mathbf{MHyp} \\ \hbox{Multihypergraphs} \end{array} \\
	&&& \begin{array}{c} \mathbf{Ctx} \\ \hbox{Formal Contexts} \end{array} &&& \begin{array}{c} \mathbf{ConLat} \\ \hbox{Concept Lattices} \end{array} \\
	\\
	\begin{array}{c} \mathbf{Hyp} \\ \hbox{Hypergraphs} \end{array} &&& \begin{array}{c} \mathbf{CoShvASC} \\ \hbox{Dowker cosheaves} \end{array} \\
	\\
	&& \begin{array}{c} \mathbf{ASC} \\ \hbox{Abstract} \\ \hbox{simplicial} \\ \hbox{complexes} \end{array} & \begin{array}{c} \mathbf{CoShvKom} \\ \hbox{Cosheaf of} \\ \hbox{chain complexes} \end{array} \\
	\\
	&&& \begin{array}{c} \mathbf{CompLat} \\ \hbox{Complete lattices} \end{array}
	\arrow["{\hbox{Lem 1}}"{description}, hook, from=2-1, to=1-4]
	\arrow["\begin{array}{c} \hbox{Collapse} \\ \hbox{Lem 3} \end{array}"{description}, shift right=5, curve={height=6pt}, from=2-1, to=5-1]
	\arrow["{\hbox{Prop 11}}"{description}, hook, from=3-4, to=1-4]
	\arrow["{\hbox{Cor 9}}"{description}, hook', from=3-4, to=2-1]
	\arrow["\begin{array}{c} \frak{B} \\ \hbox{Cor 17} \end{array}"{description}, tail reversed, from=3-4, to=3-7]
	\arrow["{\hbox{Prop 13}}"{description}, from=3-4, to=5-1]
	\arrow["\begin{array}{c} CoShvRep \\ \hbox{Def 29} \end{array}"{description}, from=3-4, to=5-4]
	\arrow["{\hbox{Thm 2}}"{description}, hook', from=3-7, to=5-4]
	\arrow["{\hbox{Thm 3}}"{description}, from=3-7, to=7-4]
	\arrow["\begin{array}{c} \hbox{Underlying} \\ \hbox{complete lattice} \\ \hbox{Prop 18} \end{array}"{description, pos=0.4}, shift left=5, curve={height=-30pt}, from=3-7, to=9-4]
	\arrow["{\hbox{Lem 3}}"{description}, shift right=5, curve={height=6pt}, hook', from=5-1, to=2-1]
	\arrow["\begin{array}{c} DH \\ \hbox{(Dowker} \\ \hbox{complex)} \\ \hbox{Lem 5} \end{array}"{description}, shift left=3, curve={height=-24pt}, from=5-1, to=7-3]
	\arrow["{\hbox{Thm 1}}"{description}, shift right=4, curve={height=30pt}, from=5-1, to=9-4]
	\arrow["\begin{array}{c} \hbox{Base} \\ \hbox{Cor 25} \end{array}"{description}, from=5-4, to=7-3]
	\arrow["\begin{array}{c} C^\Delta \\ \hbox{Lem 39} \end{array}"{description}, from=5-4, to=7-4]
\end{tikzcd}\]

\caption{Functorial diagram of the categories used in this paper: nodes are categories, and arrows indicate functors. Note that this may not be commutative.}
    \label{figures/categories}
\end{figure}

\mypng{.25}{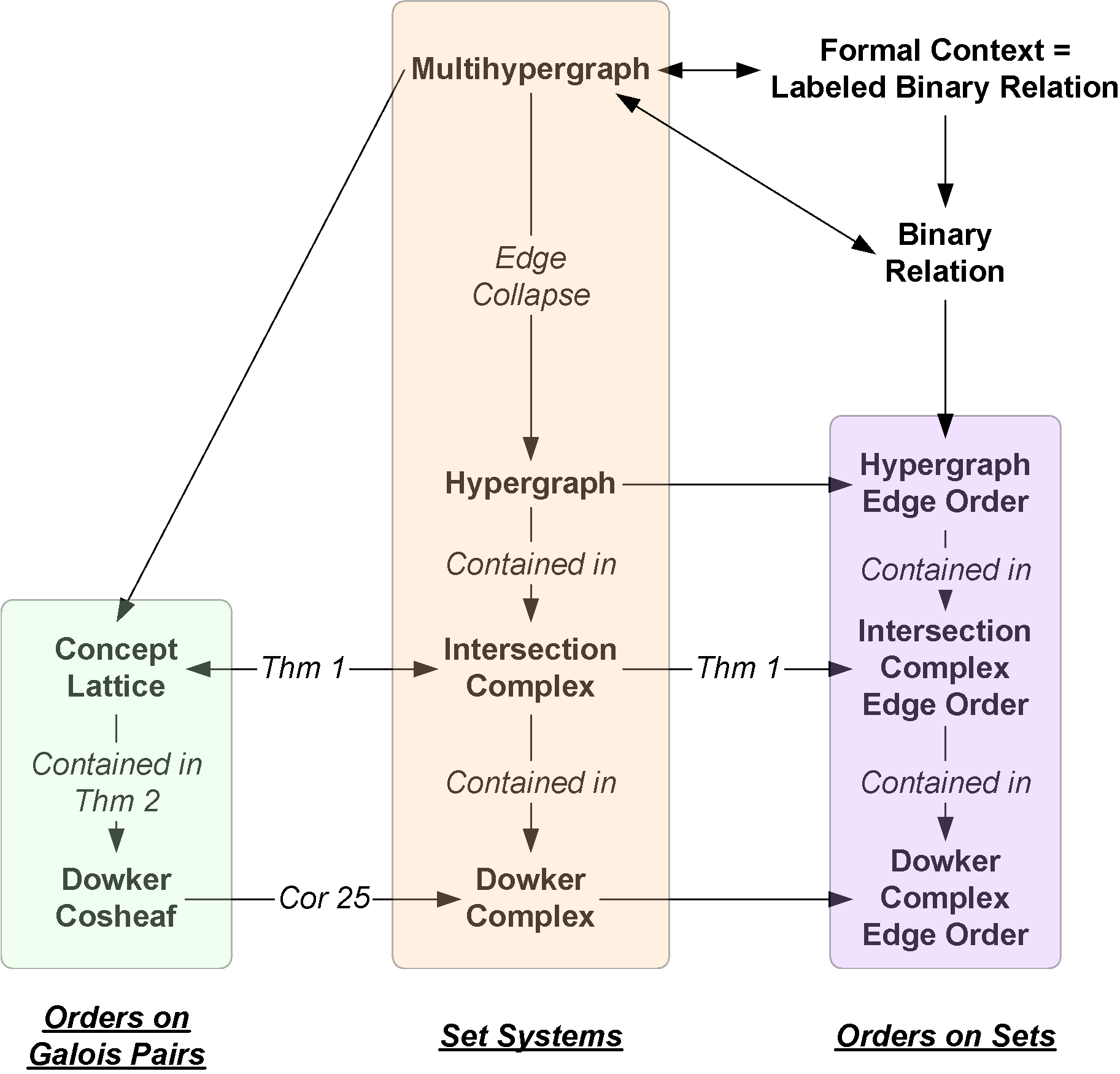}{Diagram relating the mathematical structures developed in this paper, that is, the objects of the categories from Figure \ref{figures/categories}. Nodes are types of objects, and edges indicate functional relationships, that is, the ability to construct one object uniquely from another. 
Labeled edges indicate the specific nature of a relationship and/or propositional support in the text.}


For historical context,
in the 1950s Dowker gave a way to approach binary relations through a homological lens by forming the Dowker complex of a relation \cite{Dowker1952}. 
Around 1980, FCA introduced lattice theoretical interpretations of binary relations \cite{davey_priestley_2002, GaBWiR99}.
The third perspective for taking  binary relations as incidence matrices of hypergraphs is also canonical \cite{BeC89,BrA13}. 
Recent years have seen attempts to connect these separate perspectives, for example  FCA and Dowker complexes \cite{Freund2015,ayzenberg2019topology}. More specifically, the concept lattice of a relation is sufficient to recover the homotopy invariant properties of the Dowker complex of the same relation \cite{Freund2015}.  Formal concept analysis is closely related to modern topics from topological data analysis, some of which come by way of the Dowker complex \cite{ayzenberg2019topology}.
The formal concept lattice can be obtained from the hypergraph of a relation \cite{Cattaneo2016}.  There is also a morphological relationship between hypergraphs and formal concepts  \cite{Stell2014,Bloch2017}.

This paper is organized as follows:

\begin{itemize}
    \item \sec{prelims} provides the necessary background in binary relations, hypergraphs, formal contexts, order theory, and Galois connections, all from a categorical perspective, and including a running example.
    \item \sec{sec:hypergraph_complexes}  introduces concept lattices explicitly and categorically, and proves Theorem \ref{thm:concept_iso_intclose}.
    \item \sec{sec:concept_cosheaves} establishes in Theorem \ref{thm:concept_dowker} that concept lattices are a complete invariant for the Dowker cosheaf constructed in \cite{Robinson_relations}.  Briefly, a Dowker cosheaf can be recovered up to cosheaf isomorphism from a concept lattice.
    \item \sec{sec:dowker_homology} constructs a 
    generalization of Dowker's original homological result,
    resulting in a cosheaf of chain complexes in Theorem \ref{thm:concept_iso_cosheaf_of}, whose zeroth simplicial homology is precisely the homology of the Dowker complex (Corollary \ref{cor:cosheaf_of_generalizes}).
\end{itemize}



\section{Preliminaries} \label{prelims}

We now introduce preliminary notation and ideas. \sec{sec:first} introduces the rudiments of the three perspectives on binary relations, hypergraphs, and formal concepts. This is an initial set of definitions not always  fully  explicated, but sufficient to quickly introduce  a running example in \sec{running}. \sec{categories} introduces the basics of category theory, along with a list of the categories used in the paper, while subsequent subsections of \sec{prelims} will introduce ideas from order and lattice theory, topological Dowker complexes, and Galois connections.

\subsection{Binary Relations, Hypergraphs, and Formal Contexts} \label{sec:first}

Let $R \sub X \times Y$ be a binary relation over two finite, non-empty sets $X,Y$ with $n=|X|,m=|Y|$. In general we will denote $A \sub X$ and $B \sub Y$, and we introduce an object we call a {\bf Galois pair} as an ordered pair $\rho = (A,B) \in 2^X \times 2^Y$ consisting of a subset of $A \sub X$ paired to a subset of $B \sub Y$. We will also frequently denote a Galois pair $\rho = (A,B)$ as $\gal{A}{B}$, including using $\gal{}{B}, \gal{A}{}$ when $A = \emptyset$ or $B = \emptyset$ respectively.

There are many ways to characterize a binary relation, and our primary method will be as the Boolean characteristic matrix $I_{n \times m}$ of $R$, where for $x \in X, y \in Y$,
\begin{equation*}
  I[x,y] = \halffunc{ 
		1,	& (x,y) \in R	\\
		0,	& (x,y) \not\in R	\\
  }.
\end{equation*}

This paper compares and contrasts the relationships between four different representations of a binary relation $R$ and other objects derived from them, specifically: 

\begin{itemize}
\item {\it Order theoretically} as a Galois correspondence between $X$ and $Y$, 
  
\item {\it Set theoretically} as an incidence matrix of a hypergraph over a set of vertices $X$, 
  
\item {\it Topologically} via local correspondences between $X$ and $Y$ encoded as (co)sheaves, and

\item {\it Algebraically} via local correspondences between $X$ and $Y$ encoded as chain complexes.
\end{itemize}

Towards the first end, we adapt the notation of a binary relation to that typically used in the concept lattice literature \cite{GaBWiR99}. Given a Boolean matrix $I$, derived from $R$, we let  $M=X$  now be thought of as a set of {\bf objects}, and let $G = Y$ be thought of as a set of {\bf attributes}, and call $C = (G,M,I)$ a {\bf formal context} \cite{GaBWiR99}.%
    \footnote{As also noted above, we again regret that historically different notation is used by the FCA community: whereas we'd notate a multihypergraph $H=\tup{V,E,I}$, the tradition in FCA is to describe a formal context $C=\tup{G,M,I}$ for the identical structure. Also the use of ``objects'' here is not in the sense of category theory. We will use each as appropriate, hopefully without confusion.}
Given such a context $C=(G,M,I)$, for $(g,m) \in I$ we write $gIm$ and say the object $g$ has attribute $m$. Then, for $A \subseteq G$ and $B \subseteq M$ we define:
  \begin{equation}
  \label{eq:prime_def}
  A' = \{m \in M :  gIm \text{ for all }g\in A\},	\text{ and }
    B' = \{g \in G :  gIm \text{ for all }m\in B\}.
    \end{equation}
  A \textbf{concept} of the context $(G,M,I)$ is then defined as a Galois pair $K=(A,B) = \gal{A}{B}$, where now $A\subseteq G, B \subseteq M$, and $A'=B,B'=A$.
We will call the set of all concepts of a context $\setsys{B}(G,M,I)$, or just $\cl$ when clear from context. For any concept $K=\gal{A}{B} \in \cl$,
$A$ is called its \textbf{extent} and $B$ its  \textbf{intent}.
Note that $A$ is the extent of a concept $\gal{A}{A'}$ if and only if $A'' = A$. The concepts $K \in \cl$ form a lattice $(\cl;\leq)$, 
where for $K_1=(A_1,B_1), K_2=(A_2,B_2) \in \setsys{B}(G,M,I)$, there is the order
    \({conorder} K_1 \le K_2 \h{ when } A_1 \subseteq A_2.    \)
    It then follows that $B_1 \supseteq B_2$. Such ideas from lattice theory and concept lattices in particular are  more fully explicated in \sec{sec:cl} below.

Towards the second end, we adapt the notation of binary relations to that typically used in hypergraphs. Now we let $V=X$ be thought of as a set of {\bf vertices}, and let $E = Y$ be thought of as a set of {\bf hyperedges}, and given $R$, reuse $I$ to be an {\bf incidence function} $\func{I}{V \times E}{\{0,1\}}$ of
a {\bf multi-hypergraph} $H=(V,E,I)$, where $I(v,e)=1$ iff $(v,e) \in R$. $I$ establishes $E$ as a family of hyperedges $e \in E$, where $e \sub V$, and $\forall v \in V, e \in E, v \in e$ iff $I(v,e)=1$.

Note that $H$ as described is a multi-hypergraph in that there could be duplicate hyperedges, i.e.\ $E$ is conceived of as a multiset or bag, so that it's possible for two distinct edges to be related to the exact same set of vertices. We call {\bf collapsing} the process of identifying the underlying set from the multiset $E$, that is, to remove any duplicate elements. This process renders $E$ as a set of hyperedges $E \sub 2^V$, and $H$ becomes just a {\bf hypergraph}. While not all binary relations yield incidence matrices (or formal contexts) of collapsed hypergraphs, the incidence matrix of every collapsed hypergraph is a binary relation yielding a formal context. 


\rem{
$E \subseteq \pwr(V)$,
a subset of the power set of $V$.
The {\bf dual hypergraph} $H^T = (E,V,I^T)$ is defined dually  on the transposed incidence matrix $I^T$, or the reflected relation $R^T \sub Y \times X$.
We note that a number of different hypergraph definitions are available in the literature (see e.g.\ \cite{AkSJoC20}). In our case, $H$ is not a multi-hypergraph). This construction means that edges are not duplicated in a hypergraph,
which differs sharply from the general setting of the Boolean matrix $I$ which can have duplicated rows and columns. 
}

Finally, by focusing our attention on certain appropriate subsets of $X$ or $Y$,
we will obtain the topological, sheaf-theoretical perspective. This will be fully introduced  below in Sections \ref{sec:background_setsystems} and \ref{sec:background_cosheaf}.


\subsection{Running Example} \label{running}
  \label{eg:running}

\rem{

  We will use a running example over $X = \{ a,b,c,d\}, Y = \{ 1,2,3,4,5,6 \}$, as shown as an incidence matrix in Table~\ref{tab:running}. 
  \rawson{This is not a hypergraph per our definition!}
  The pair of dual hypergraphs are shown in \fig{fig:hyper_ex}. We can also list the formal concepts, using compact set notation we will suggestively call ``Galois notation'' (for example, the concept $K=(\{a,b\},\{1,4\})$ is simplified to $ab:14$), as
  \begin{equation*}
    \setsys{B}(G,M,I) = \{ abc:1, ab:14, bd:3, cd:2, b:134, c:12, d:2356 \}.
  \end{equation*}
  \rawson{Add $\emptyset:123456$, $abcd:\emptyset$, ...}

}

We will use a running example over $X = \{ a,b,c,d\}, Y = \{ 0,1,2,3,4,5 \}$, as shown as an incidence matrix in Table~\ref{tab:running}. 
  The multihypergraph is shown on the left side of  \fig{fig:hyper_ex} as an {\bf Euler diagram}, where a closed curve represents the hyperedge surrounding a collection of points as its vertices. Noting that columns 0 and 2 are equal, yielding $e_0=e_2=\{a,d\}$, we show the collapsed hypergraph on the right, and imagine the collapsed matrix missing column 0. 

\begin{table}
  \begin{center}
    \caption{Running example binary relation.}
    \begin{tabular}{ c||c|c|c|c|c|c }
       & 0 & 1 & 2& 3& 4& 5\\
     \hline
      a   & x & x    & x  &    &   & x\\
      b   &&     &   &   x&   x&x \\
      c   &&  x   &  &    & x   &x \\
      d   &x&    x  & x &   &    & \\
    \end{tabular}
    \label{tab:running}
  \end{center}
\end{table}

\begin{figure}
  \begin{center}
    \begin{tabular}{ccc}
      \subfloat{\includegraphics[width = 3in]
      {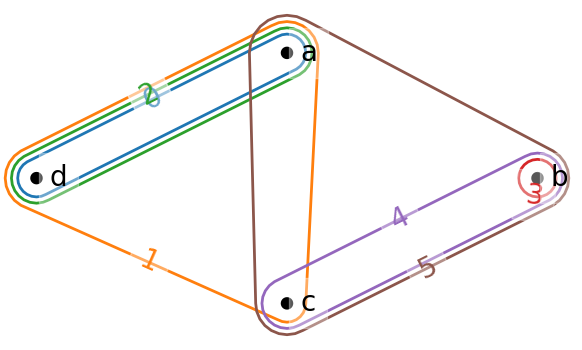} } &
      \subfloat{\includegraphics[width = 3in]{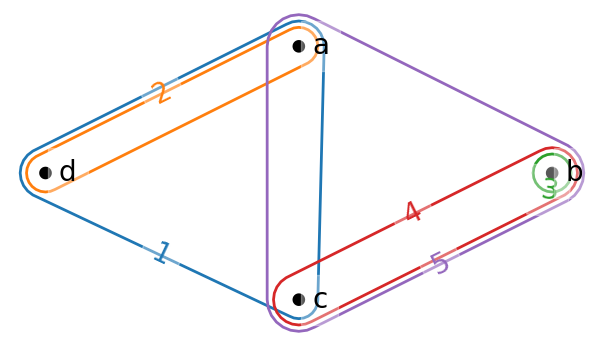}} 
    \end{tabular}
    \caption{(Left) Euler diagram of a multihypergraph $H$ in Example \ref{eg:running}. (Right) Its collapsed hypergraph.}
    \label{fig:hyper_ex}
  \end{center}
\end{figure}

Recalling the notational convention that $G=X,M=Y$, then the 10 formal concepts are:
  \begin{equation*}
    \setsys{B}(G,M,I) = \{ \gal{}{012345}, \gal{b}{345}, \gal{c}{145}, \gal{a}{0125}, \gal{bc}{45}, \gal{ac}{15},  \gal{ad}{012}, \gal{abc}{5}, \gal{acd}{1}, \gal{abcd}{}  \}.
\end{equation*}
    Note how our Galois pair notation introduced earlier is efficient when identifers are single characters used in compact set notation. For example, the concept $K=(\{a,c\},\{1,5\})$ is simplified to $K = \gal{ac}{15}$.

The Hasse diagram of the concept lattice for our example is then shown in \fig{figures/H_final_CL2}. On the left side the ``full context lattice'' is shown, with for each concept $K=(A,B) = \gal{A}{B} \in \cl$ with intent $A \sub G$ and extent $B \sub M$ shown as a Galois pair. Note that starting here, and through the sequel, for visual clarity we show the left (extent) member of the Galois pair in black, and the right (intent) in red. A reduced version is shown on the right-hand side, a term which will be fully introduced in \sec{sec:cl} below.

\mypng{.5}{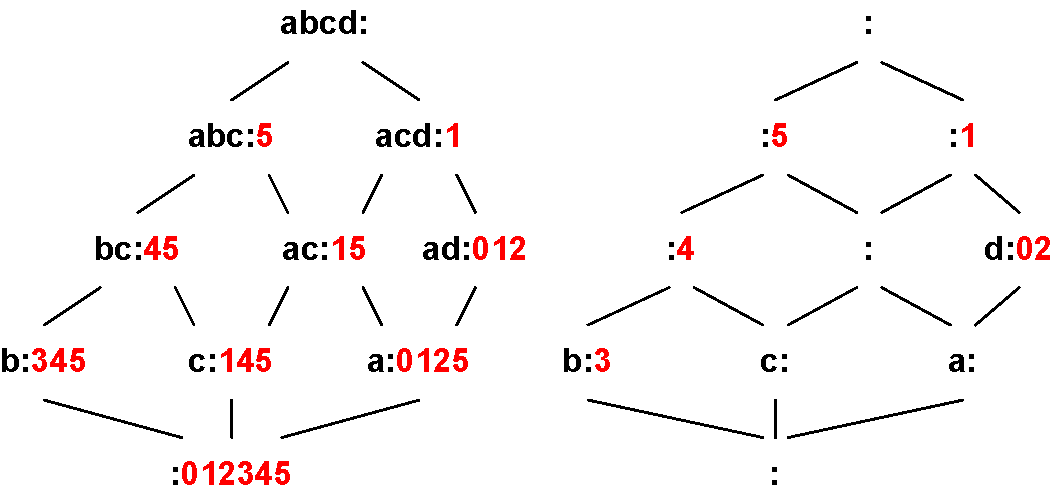}{(Left) The Hasse diagram of the full concept lattice for our example. Concepts are nodes, Galois pairs of  extents (black, left) and intents (red, right). (Right) Its reduced form (see \sec{sec:cl}).
\label{fig:concept_lattice}
}

\subsection{Categories} \label{categories}

This paper uses the mathematics of category theory to provide the organizing structure.
The main insight of category theory is that in addition to mathematical objects, such as sets, hypergraphs, and formal contexts,
one should study classes of structure preserving transformations of these objects.
While familiarity with category theory is not required to grasp the main ideas of this paper,
a few preliminaries will help.

\begin{definition}
  A \textbf{category} $\cat{C}$ consists of:
  \begin{itemize}
  \item A class $Ob(\cat{C})$ of \textbf{objects},
  \item A class $Mor(\cat{C})$ of \textbf{morphisms}, and
  \item A binary operation $\circ$ called \textbf{composition} on certain pairs of morphisms, as below.
  \end{itemize}
  Each morphism $m \in Mor(\cat{C})$ can be expressed as an arrow $m : A \to B$ where $A, B \in Ob(\cat{C})$.
  Moreover, if $m_1 : A_1 \to A_2$ and $m_2 : A_2 \to A_3$ are morphisms, their composition uniquely defines $(m_2 \circ m_1) : A_1 \to A_3$.

  Additionally, the following axioms must hold:
  \begin{itemize}
  \item Each object $A$ has a unique $\id_A : A \to A$ which is an identity when composed with other morphisms on the left $m \circ \id_A = m$ or right $\id_A \circ m = m$, and
  \item The composition is an associative operation.
  \end{itemize}

  Two objects $A$ and $B$ are said to be \textbf{isomorphic} in $\cat{C}$ if there exist two morphisms $p: A \to B$ and $q: B \to A$ such that $q \circ p = \id_A$ and $p \circ q = \id_B$.
  In this case, $p$ and $q$ are called \textbf{isomorphisms}.
\end{definition}

While the abstraction of the definition of a category may seem excessive, categories abound.
The precision afforded to the use of categories allows one to control the data types and their transformations quite explicitly.
This becomes important in a number of places throughout our study. 
The relationships between concept lattices and their various topological instantiations are subtle, 
and some of the apparently ``obvious'' relationships are not quite true.  
Categories help elucidate the correct relationships in these cases.
The reader is therefore cautioned that some of the categories have straightforward definitions that are essentially the only option possible, whereas others require specific choices in their definition.

Mathematical analogies are best captured by the notion of a functor, a kind of structure-preserving transformation of categories themselves.

\begin{definition}
  A \textbf{covariant functor} $F$ from one category $\cat{C}$ to another $\cat{D}$ is written $F: \cat{C} \to \cat{D}$ and consists of the following data:
  \begin{itemize}
  \item A function on objects, usually written as $F : Ob(\cat{C}) \to Ob(\cat{D})$, and
  \item A function on morphisms $Mor(\cat{C})$ to $Mor(\cat{D})$, usually also written with the symbol $F$, since no confusion can arise as objects and morphisms are of different types, $F : Mor(\cat{C}) \to Mor(\cat{D})$,
  \end{itemize}
  subject to two constraints:
  \begin{enumerate}
  \item If $m: A \to B$ is a morphism in $\cat{C}$, then $(Fm) : F(A) \to F(B)$ is a morphism in $\cat{D}$, and
  \item $F$ respects composition: $F(m_2 \circ m_1) = F(m_2) \circ F(m_1)$.
  \end{enumerate}

  A \textbf{contravariant functor} consists of the same data with the change that it reverses the direction of morphisms.
  Namely, $(Fm) : F(B) \to F(A)$ and $F(m_2 \circ m_1) = F(m_1) \circ F(m_2)$ for a contravariant functor.

  Without further qualification, ``functor'' will be taken to mean a covariant functor.
\end{definition}

For easy reference, this paper considers the following categories, which are defined in the subsequent sections, and among whom the proposed functors are shown in \fig{figures/categories}:

\begin{description}

\item[$\cat{CompLat}$] (Definition \ref{def:complat}) The category of complete lattices.

\item[$\cat{Rel}$] (Definition \ref{def:rel} ) The category of binary relations.

\item[$\cat{MHyp}$] (Definition \ref{def:mhyp}) The category of multi-hypergraphs derived from an incidence relation.

\item[$\cat{Hyp}$] (also Definition \ref{def:hyp}) The category of (collapsed) hypergraphs, which is a full subcategory of $\cat{MHyp}$.

\item[$\cat{Asc}$] (Definition \ref{def:asc}) The category of abstract simplicial complexes.

\item[$\cat{Ctx}$] (Definition \ref{def:ctx}) The category of formal contexts which captures the constraint between particular elements of a binary relation so as to preserve the formal concepts.

\item[$\cat{ConLat}$] (Definition \ref{def:conlat}) The category of concept lattices, which we also introduce here.

\item[$\cat{CoShvAsc}$] (Definition \ref{def:coshvasc}) We introduce the category of cosheaves of abstract simplicial complexes.


\item[$\cat{CoShvKom}$] (Definition \ref{def:coshvkom}) We introduce the category of cosheaves of chain complexes.

\end{description}

\subsection{Order Theory} \label{order-sec}

See \cite{davey_priestley_2002} for more details.

\begin{definition}
  Let $P$ be a set. A \textbf{partial order} on $P$ is a binary relation $\leq$ on $P$ such that for all $x,y,z \in P$, 
  \begin{enumerate}
  \item $x \leq x$,
  \item $x \leq y$ and $y \leq x$ implies $x = y$, 
  \item $x \leq y$ and $y \leq z$ implies $x \leq z$.
  \end{enumerate}
  The pair $(P,\leq)$ is called a \textbf{partially ordered set} or a \textbf{poset}.
  When $\leq$ is clear from context, we will usually say ``$P$ is a poset'' without ambiguity. When $x \le y$ we also denote $y \ge x$, and if additionally $x \neq y$ then we can say $x < y$ and $y > x$. For $S \sub P$, then $(S,\le_S)$ is a poset where $\le_S$ is $\le$ restricted to $S$.
\end{definition}

Every poset $P$ also has a natural dual poset $P^{op}$ where $y\leq x$ in $P^{op}$ if and only if $x\leq y$ in $P$.
Given a poset $P$ it is common to display it visually as a {\bf Hasse diagram}, which shows only its reflexive and transitive reduction. Nodes in the diagram are elements of $P$, with $x$ vertically above $y$, and a line segment descending from $x$ to $y$, when $x$ ``covers'' $y$ in that $x > y$ and there is no $z$ with $x > z > y$. We will use Hasse diagrams routinely, starting with the example below.

\begin{definition}[Edge Poset]
  Given a hypergraph $H=(V,E)$, construct its \textbf{edge poset} as $(E,\subseteq)$.
  \end{definition}

Figure \ref{figures/hg_hasse} shows the Hasse diagram for the edge poset of the collapsed hypergraph $H$ in Example \ref{eg:running}. Note that we only have two disjoint chains, with $\{a,d\} \sub \{a,c,d\}$ and $\{b\} \sub \{b,c\} \sub \{a,b,c\}$, where the figure shows compact set notation.

\mypng{.5}{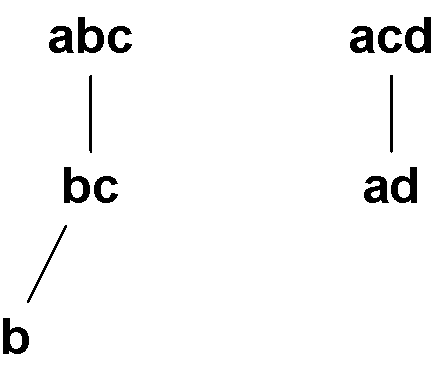}{The Hasse diagram of the poset $(E,\sub)$ for our example, using compact set notation.}

\rem{

\begin{figure}
  \begin{center}
    \begin{equation*}
      \begin{tikzcd}
    & {[a,b,c]} \\
	        {[a,b] } & {[a,c]} & {[b,c]} & {[b,d]} & {[c,d] } \\
	        & {[a]} & {[b]} & {[c] } & {[d]}
	        \arrow[from=3-2, to=2-1]
	        \arrow[from=3-2, to=2-2]
	        \arrow[from=3-3, to=2-1]
	        \arrow[from=3-3, to=2-3]
	        \arrow[from=3-3, to=2-4]
	        \arrow[from=3-5, to=2-4]
	        \arrow[from=3-4, to=2-3]
	        \arrow[from=3-4, to=2-5]
	        \arrow[from=3-5, to=2-5]
	        \arrow[from=3-4, to=2-2]
	        \arrow[from=2-2, to=1-2]
	        \arrow[from=2-3, to=1-2]
	        \arrow[from=2-1, to=1-2]
      \end{tikzcd}
    \end{equation*}
    \caption{Hasse diagram of the face poset of $D(G,M,I)$ for the context defined in Example \ref{eg:running}.}  
    \label{fig:running_face_poset}
  \end{center}
\end{figure}

}

\begin{definition}
  Given two partially ordered sets $(P, \leq_P)$ and $(Q,\le_Q)$,
  a function $f: P \to Q$ is \textbf{order preserving} if for all $x,y \in P$,
  $x \leq_P y$ implies that $f(x) \le_Q f(y)$.
  If this condition holds more strongly, namely if for all $x,y \in P$,
  it follows that $x \leq_P y$ if and only if $f(x) \le_Q f(y)$,
  then we say that $f$ is \textbf{order embedding}.
  
  If $f$ is also surjective then it is an \textbf{order isomorphism}.
\end{definition}

\begin{definition}
  The category $\cat{Pos}$ has every partially ordered set $(P,\leq)$ as an object.
  Each morphism $g : (P,\leq_P) \to (Q,\leq_Q)$ in $\cat{Pos}$ consists of an order preserving function $g : P \to Q$ such that if $x$ and $y$ are two elements of $P$ satisfying $x \leq_P y$,
  then $g(x) \leq_Q g(y)$ in $Q$.
  Morphisms compose as functions on their respective sets.
\end{definition}

\begin{definition}
  Let $P$ be a nonempty poset. An element $x \in P$ is \textbf{minimal} if $\not \exists y \in P, y < x$, and let $\min(P) \sub P$ be the set of its minimal elements. An element $x \in P$ is \textbf{maximal} if $\not \exists y \in P, y > x$, and let $\max(P) \sub P$ be the set of its maximal elements. For $x \neq y \in P$, let 
  \begin{equation*}
    J(x,y) = \min(\{ z \in P \st z \ge x, z \ge y \}) \sub P,	\quad
    M(x,y) = \max(\{ z \in P \st z \le x, z \le y \}) \sub P
  \end{equation*}
  be their \textbf{join set} and \textbf{meet set} respectively. By  notational extension, let $J(S),M(S)$ be the join and meet sets over a non-empty collection of nodes $S \sub P$. When $|J(x,y)|=1$ (resp.\ $|M(x,y)=1|$) then let $x \vee y \in J(x,y)$ ($x \wedge y \in M(x,y)$) be their unique \textbf{join} and \textbf{meet}, and similarly for $\bigvee S, \bigwedge S$.
  \begin{enumerate}
  \item If $x \vee y$ and $x \wedge y$ exist for all $x$ and $y$ in $P$ then $P$ is a \textbf{lattice}. 
  \item If additionally $\bigvee S$ and $\bigwedge S$ exist for all $S \subseteq P$ then $P$ is called a \textbf{complete lattice}. 
  \end{enumerate}
\end{definition}

\begin{definition}
  Let $L$ be a lattice and $\emptyset \neq M \subseteq L$ then $M$ is a \textbf{sublattice} of $L$ if $a,b \in M$ implies $a \vee b \in M$ and $a \wedge b \in M$.
\end{definition}

\begin{definition}
  If $L$ and $K$ are lattices then an order preserving $f : L \to K$ is a \textbf{lattice homomorphism} if for all $a,b\in L$,
  \begin{equation*}
    f(a\vee b) = f(a) \vee f(b), \quad 
    f(a\wedge b) = f(a)\wedge f(b).
  \end{equation*}
  If $f$ is a bijective homomorphism then it is a \textbf{lattice isomorphism}.
  If $f: L \to K$ is a one to one homomorphism the sublattice $f(L)$ of $K$ is isomorphic to $L$ and we refer to it as a \textbf{lattice embedding}. 
\end{definition}

It is a straightforward exercise to demonstrate that the composition of two lattices homomorphisms $f_1: L_1 \to L_2$ and $f_2: L_2 \to L_3$ yields a function $(f_2 \circ f_1): L_1 \to L_3$ that is itself a lattice homomorphism.  This means that the collection of all lattices and lattice homomorphisms forms a mathematical category.

\begin{definition} \cite{nlab:lat}
  \label{def:complat}
  The category $\cat{CompLat}$ consists of all complete lattices, with all associated lattice homomorphisms as morphisms.
\end{definition}

\begin{definition}
  Let $P$ be an ordered set and let $S \subset P$.
  Then $S$ is called \textbf{join-dense} in $P$ if for every element $x \in P$ there is a subset $A$ of $S$ such that $x = \bigvee_{P}A$.
  The dual of join-dense is \textbf{meet-dense}.
\end{definition}


\subsection{Hypergraphs and Dowker Complexes}
\label{sec:background_setsystems}

The categorical perspective on hypergraphs can be traced back to \cite{Drfler1980} and has more recently been furthered, and in some cases corrected, in \cite{Grilliette2018IncidenceHT} and \cite{grilliette2024simplificationincidenceincidencefocused}. The authors will use a slightly different definition of the category of multi-hypergraphs which we prove yields an isomorphic category. This is done to make more transparent the connection to the other categories at play in this paper. We will first start with the category of relations as it is a natural starting point. 

\begin{definition}[The category $\cat{Rel}$]  (see \cite{brun2019sparse} or \cite[Def. 6]{Robinson_relations})
    \label{def:rel} 
$\cat{Rel}$ has triples $(X,Y,I)$ for objects, where $X$ and $Y$ are sets, and $I \subseteq X \times Y$.  A morphism $(X_1,Y_1,I_1) \to (X_2,Y_2,I_2)$ in $\cat{Rel}$ is defined by a pair of functions $f:X_1\to X_2$, $g:Y_1\to Y_2$ such that $I_2(f(x),g(y))=1$ whenever $I_1(x,y)=1$.  Composition of morphisms is simply the composition of the corresponding pairs of functions, which means that $\cat{Rel}$ satisfies the axioms for a category.
\end{definition}

The appropriate category $\cat{MHyp}$ for multi-hypergraphs is derived from $\cat{Rel}$ by restricting the class of morphisms in the following way. (Note when using $e'$ below we are using the same $'$ defined in section 2.1 but on the singleton set $\{e\}$.)

\begin{definition}[The category $\cat{MHyp}$]
\label{def:mhyp}
The objects of $\cat{MHyp}$ are the same as $\cat{Rel}$. 
A morphism $f: H_1 \to H_2$ from $H_1 = (V_1,E_1, I_1)$ to $H_2 = (V_2,E_2,I_2)$ in $\cat{MHyp}$ is a pair of functions $f_V : V_1 \to V_2$ and $f_E: E_1 \to E_2$ such that for each $e \in E_1$, 
\begin{equation*}
    f_V(e') = (f_E(e))',
\end{equation*} 
where we are using the abbreviated notation 
\begin{equation*}
  f_V(e') = \{ f_V( v ) : v \in e' \}
\end{equation*}

In other words $f_V$ maps the vertices of an edge $e$ surjectively to the vertices of its image hyperedge $f_E(e)$.
\end{definition}

While the objects of $\cat{Rel}$ and $\cat{MHyp}$ and $\cat{Ctx}$ (defined below) are the same, we suggestively use $(X,Y,I)$ for relations, $(V,E,I)$ for multihypergraphs and $(G,M,I)$ for contexts since this is common in their respective literature.

  \begin{definition}
    \label{def:hyp}
  Suppose that $H_1 = (V_1,E_1)$ and $H_2 = (V_2,E_2)$ are two hypergraphs.
  A {\bf hypergraph morphism} $m : H_1 \to H_2$ is a function on vertices $m : V_1 \to V_2$ such that for every $e_1 \in E_1$ there is an $e_2 \in E_2$ such that $m(e_1)= e_2$.
  The {\bf composition} $m_2 \circ m_1$ of two hypergraph morphisms $m_1 : H_1 \to H_2$ and $m_2 : H_2 \to H_3$ is defined by the composition of their vertex functions.  
\end{definition}

\begin{lemma}
  Composition of hypergraph morphisms is well-defined.
  If $m_1 : H_1 \to H_2$ and $m_2 : H_2 \to H_3$ are hypergraph morphisms, then $(m_2 \circ m_1) : H_1 \to H_3$ is itself a hypergraph morphism.
\end{lemma}
\begin{proof}
  If $e_1 \in E_1$, then there is an $e_2 \in E_2$ such that $m_1(e_1) = e_2$.
  Continuing this idea, there is also an $e_3 \in E_3$ such that $m_2(e_2) = e_3$.
  Therefore, $m_2(m_1(e)) =  m_2(e_2) \subseteq e_3$ as desired.
\end{proof}

Since the objects of $\cat{Rel}$ and $\cat{MHyp}$ are the same and the morphisms of $\cat{MHyp}$ are just a more restricted subclass of those in $\cat{Rel}$ it is obvious that $\cat{MHyp}$ is in fact a subcategory of $\cat{Rel}$ 

The authors will note that our category of multi hypergraphs is not presented in the typical way; however, it fits more cleanly in our language and is isomorphic to a standard definition. In \cite{Grilliette2018IncidenceHT} the main category of set system hypergraphs, denoted $\mathfrak{H}$ and comparable to our category $\cat{MHyp}$, is defined as follows.
\begin{definition} \cite{Grilliette2018IncidenceHT} 
    The category $\mathfrak{H}$ is given by the following:
    \begin{itemize}
        \item The objects are triples $H=(V,E,\epsilon)$ where $V,E$ are sets and $\epsilon: E \to \mathcal{P}(V)$, whre $\mathcal{P}$ is the power set. The sets $V$ and $E$ are referred to as the vertex and edge set respectively. 
        \item A morphism $(V,E,\epsilon) \to (V',E',\epsilon')$ is a pair $(f_V,f_E)$ where $g: V \to V'$ and $f: E \to E'$ are morphisms in set such that the diagram in Figure \ref{fig:hypergraph_morphism_def} commutes.
    \end{itemize}
\begin{figure}
    \begin{center}
\begin{tikzcd}
E \arrow[r, "f_E"] \arrow[d, "\epsilon"'] & E' \arrow[d, "\epsilon'"] \\
\mathcal{P}(V) \arrow[r, "\mathcal{P}(f_V)"'] & \mathcal{P}(V')
\end{tikzcd}
\end{center}
    \caption{The defining diagram for a morphism in $\mathfrak{H}$.}
    \label{fig:hypergraph_morphism_def}
\end{figure}

\end{definition}

\begin{lemma}
    $\cat{MHyp}$ is isomorphic to $\mathfrak{H}$
\end{lemma}

\begin{proof}
    Define a functor $F : \cat{MHyp}\to \mathfrak{H}$ by the following recipe:
    \begin{description}
        \item[Objects]: The object $(V,E,I)$ in $\cat{MHyp}$ is taken to
        \begin{equation*}
            F(V,E,I) := (V,E,\epsilon),
        \end{equation*}
        where
        \begin{equation*}
            \epsilon(e) := \{v\in V| I(v,e)=1\}.
        \end{equation*}
        \item[Morphisms]: Given a morphism $f$ from $H \to G$ in $\cat{Mhyp}$, this consists of a pair of functions $f_V$ and $f_E$ on the vertex and edge sets, respectively.  Let $F$ simply preserve these two functions, which means that the morphism in $\mathfrak{H}$ is also defined by $f_V$ and $f_E$.  To see that this is well-defined morphism in $\mathfrak{H}$, note that the condition that $f_V(e) = V(f_E(e))$ for all $e\in E_G$ is exactly the condition that $\mathcal{P}(f_v)(\epsilon_H(e)) = \epsilon_G((f_E(e))$.
    \end{description}

    To establish that $F$ is an isomorphism of categories, we establish that it is a bijection both on objects and morphisms.

    $F$ gives a bijection of objects as the sets for the vertices and the edges stay the same and all that changes is the $I$ or $\epsilon$ respectively. 
    
    $F$ is also bijective on morphism sets.  The well-definedness condition noted above is exactly the same condition as the diagram in Figure \ref{fig:hypergraph_morphism_def} commuting. 
    As a result, $f$ is a morphism in $\cat{MHyp}$ if and only if it is a morphism in $\mathfrak{H}$.
\end{proof}

This isomorphism also allows us to use the following result which has been rephrased from \cite{grilliette2024simplificationincidenceincidencefocused} into our language.

\begin{lemma}
    There exists a functor which we will call $Collapse$ from $\cat{MHyp}$ to $\cat{Hyp}$ defined as follows. If $(V,E,I)$ is an object in $\cat{MHyp}$ then $Collapse((V,E,I)) = (V,E_c)$ where $E_c\subseteq \mathcal{P}(V)$ and $h\in E_c$ if there exists some $e\in E$ such that $e' = h$. If $(f_v,e_v)$ is a morphism in $\cat{MHyp}$ then $Collapse(f_v,f_e) = f_v$.
    Moreover, this functor is actually right adjoint to the natural inclusion of categories going the other way.
    
\end{lemma}

\begin{definition}
  A hypergraph $H = (V,E)$ is called {\bf topped} if $V \in E$, and {\bf bottomed} if $\emptyset \in E$.
\end{definition}

Topped hypergraphs correspond to $Collapse$ images of formal contexts $(V,E,I)$ with a ``full'' column in $I$, while bottomed will have an empty column.

\begin{definition}[Abstract Simplicial Complex (ASC)] 
  Given a hypergraph $H=(V,E)$ we say that an edge $e \in E$ is a \textbf{abstract simplex} if for every nonempty $e' \subseteq e$ we have that $e' \in E$.
  An \textbf{abstract simplicial complex (ASC)} is a hypergraph $H^\sub = (V,E^\sub)$ where every edge in $E$ is an abstract simplex. 
  Explicitly  for every $e \in E$ and $e' \subseteq e$ then $e' \in E$.
\end{definition}

    \[ E^\sub = \{ f \sub e \}_{e \in E} \]
 We will use square brackets, like $[a,b,c,d]$, to mean the minimal ASC which has $\{a,b,c,d\}$ as a face which is maximal in the subset order,
so that $[a,b,c,d]=\pwr(\{a,b,c,d\})$. And if $e=\{a,b,\ldots,z\} \sub V$ is a set, then use $[e] = [a,b,\ldots,z]$. From here out we will use symbols like $\Sigma$ and $\Xi$ to distinguish ASCs from other hypergraphs.

\begin{definition}[Simplicial Map] \cite{spanier,Hatcher_2002}
  \label{def:asc}
   Suppose $\Sigma$ and $\Xi$ are simplicial complexes with vertex sets $V_\Sigma$ and $V_\Xi$ respectively.
  A function $f : V_\Sigma \to V_\Xi $ on vertices is called a \textbf{simplicial map} $f : \Sigma \to \Xi$ if it transforms each simplex $[v_0,..,v_k]$ of $\Sigma$ into a simplex $[f(v_0),...,f(v_k)]$ of $\Xi$, 
  after removing any duplicate vertices.
  The category $\cat{Asc}$ has ASCs as its objects and simplicial maps as it morphisms. 
\end{definition}

\begin{lemma}
    $\cat{Asc}$ is a full subcategory of $\cat{Hyp}$  
\end{lemma}

\begin{proof}
    To prove this we need to show that every simplicial map between two ASCs inside of $\cat{Hyp}$ is a morphism in $\cat{Hyp}$, and also that every morphism in $\cat{Hyp}$ between ASCs is a simplicial map. Given two ASCs $\Sigma$ and $\Xi$ it is obvious that every simplicial map is a morphism in $\cat{Hyp}$ as a simplicial map is a stronger criteria. To show that every morphism from $\Sigma$ to $\Xi$ is a simplicial map we just need to see that every hyperedge in both is by definition a simplex and thus every simplex in $\Sigma$ must get sent to a simplex in $\Xi$ if every hyperedge is sent to a hyperedge.
\end{proof}

\begin{definition}
\label{def:hyprep}
Suppose that $(G,M,I)$ is a formal context.
This can be used to construct a hypergraph $HypRep(G,M,I)=(V,E)$ with $V=G$, whose edges $E$ are indexed by the elements $m\in M$.
Specifically, each hyperedge $e_m = \{m\}'$ is the extent of some $m \in M$.
Duplicates, if they occur, only count once.
\end{definition}

We will show later (in Proposition \ref{prop:ctx_to_hyp} that $HypRep$ is functorial.  Nevertheless, with Definition \ref{def:hyprep} in hand, we can convert formal contexts into ASCs.

\begin{definition}[Dowker Complex]
  \label{def:dowker_complex}
  Given a hypergraph 
  $H = (V,E)$, then its \textbf{Dowker complex} 
  $DH(V,E)$ is the ASC: 
  \begin{equation*}
    E^D = \bigcup_{e \in E} [e].
  \end{equation*}

  We will usually construct Dowker complexes from a context instead of a hypergraph.
  so we will use the notation
  \begin{equation*}
    D(G,M,I) := (DH \circ HypRep)(G,M,I).
  \end{equation*}
\end{definition}

\begin{lemma} (see also \cite{purvine_Hgraph_homology})
$DH : \cat{Hyp} \to \cat{Asc}$ is a covariant functor.
\end{lemma}

\begin{proof}
First of all, given a hypergraph $(V,E)$, which is an object in $\cat{Hyp}$, the Dowker complex $DH(V,E)$ is closed under subsetting by construction.  Thus, $DH(V,E)$ is an object of $\cat{Asc}$. 

Now suppose that $m: (V_1,E_1) \to (V_2,E_2)$ is a hypergraph morphism.  By definition, this means that $m$ corresponds to a function on vertices $m: V_1\to V_2$.  On the other hand, a simplicial map is also a function on vertices.  This defines the action of $DH$ on morphisms.  We show that this is well-defined, and then that it preserves composition covariantly.

Suppose that $[v_0, \dotsc, v_k]$ is a simplex of $DH(V_1,E_1)$.  That means that $\{v_0,\dotsc,v_k\}$ is a subset of some $e_1 \in E_1$.  Since $m$ is a hypergraph morphism,  $m(e_1) \in E_2$.  On the other hand, $m(\{v_0, \dotsc, v_k\}) = \{m(v_0),\dotsc,m(v_k)\}$ is a subset of $m(e_1)$, so $[m(v_0),\dotsc,m(v_k)]$ is a simplex of $DH(V_2,E_2)$.

Finally, notice that composition of hypergraph morphisms is achieved by composing the corresponding vertex functions.  This is exactly the same operation, in the same order, for simplicial maps.  Thus $DH$ preserves the composition fo morphisms.
\end{proof}

\rem{

\begin{definition}
    Let $E^\sub := \bigcup_{f : f \sub e \in E} f$.
\end{definition}

}

   \fig{asc} shows the Dowker complex $DH(V,E)$ for our example, on the left as an Euler diagram, and on the right as its face (edge) order, using compact set notation. 
    
\begin{figure}
  \begin{center}
    \begin{tabular}{cc}
       \subfloat{\includegraphics[width = 3in]      {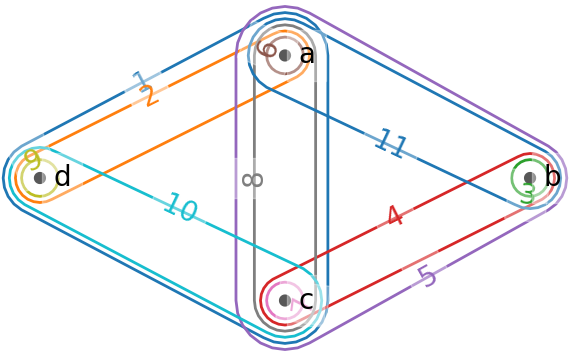} } &
      \subfloat{\includegraphics[width = 3in]{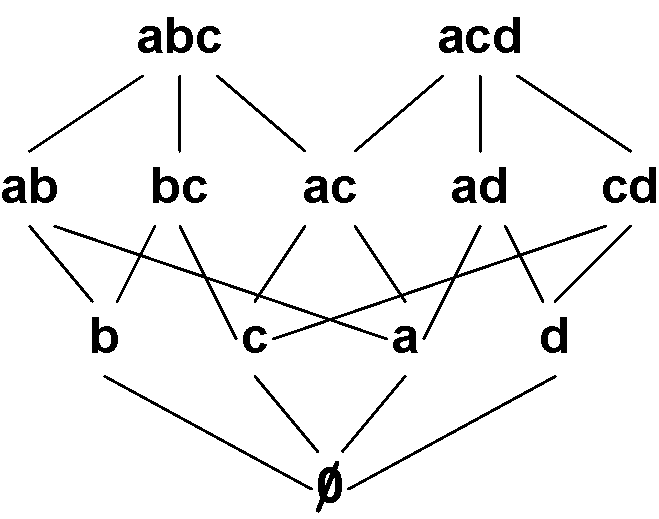}} 
    \end{tabular}
    \caption{(Left) Euler diagram of the Dowker complex $DH(V,E)$ of our example hypergraph. (Right) Its face (edge) order. 
 }
    \label{asc}
  \end{center}
\end{figure}

\rem{

\mypng{.375}{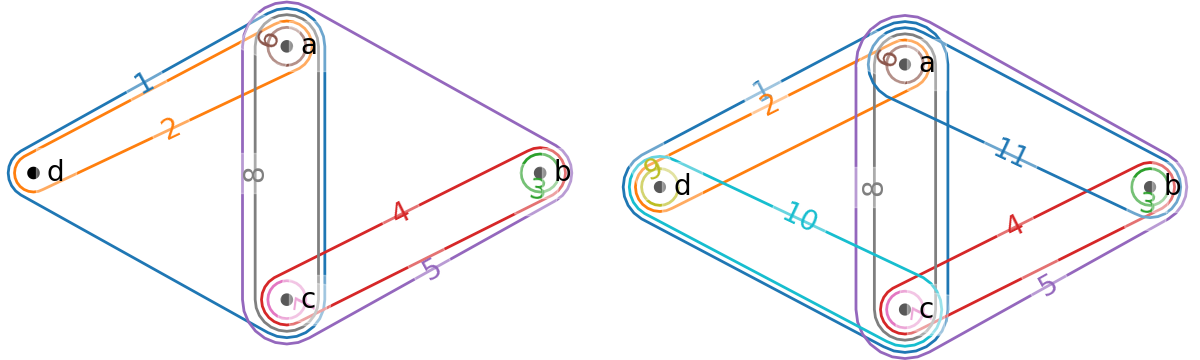}{(Left) The intersection complex $H^\cap$ for our example. (Right) The ASC $H^\Delta$.}

}

\subsection{Galois Connections: Polar Maps and Closure Operators}




Return now to FCA notation and concepts, where we have a formal context $(G,M,I)$ of objects $G$ and attributes $M$. 
Recall then that an object set $A \sub G$ can be the extent of a formal concept if and only if $A'' = A$, where here we use the prime notation from \equ{eq:prime_def}. This type of requirement actually is a special kind of more general closure operator.

\begin{definition}
  Let $P$ be a poset.
  A map $c: P \to P$ is called a \textbf{closure operator on $P$} if, for all $x,y\in P$,
  \begin{enumerate}
    \item $x \leq c(x)$,
    \item $x \leq y$ implies $c(x) \leq c(y)$, and
    \item $c(c(x)) = c(x)$.
  \end{enumerate}
  We also say that $c$ is a \textbf{closure operator on a set $X$} if it is a closure operator on its power set $\pwr(X)$ taken as a poset with the subset order.
\end{definition}

To identify the closure operator provided by the concept lattice operator $\bullet'$, we first  observe that  definition  (\ref{eq:prime_def}) presents $\bullet'$ as  a pair of maps between the power sets $\mathcal{P}(G)$ and $\mathcal{P}(M)$ of objects and attributes taken as posets. 
And additionally,  (\ref{conorder}) defines the order on concepts with subsets for extents and supersets for intents. This implies these are actually order preserving maps from $\mathcal{P}(G)$ to $\mathcal{P}(M)^{op}$.

\begin{definition}
  \label{def:polar}
  Let $(G,M,I)$ be a context.  The \textbf{polar maps} $I^* : \pwr(G) \to \pwr(M)^{op}$ and $I_* :\pwr(M)^{op} \to \pwr(G)$ are given by

  \[
    I^*(A) := A' = \{m \in M : gIm \text{ for all }g\in A\} \]
\[    I_*(B) := B' = \{g \in G : gIm \text{ for all }m\in B\}.
  \]
\end{definition}

These maps are a special type of map between two posets known as a Galois connection.

\begin{definition}
  Let $P$ and $Q$ be ordered sets.
  A pair $(f^*, f_*)$ of maps $f^* : P \to Q$ and $f_* : Q \to P$ (called \textbf{right} and \textbf{left} respectively) is
  a \textbf{Galois connection} between $P$ and $Q$ if, for all $p\in P$ and $q\in Q$, $f^*(p) \leq q$ if and only if $p \leq f_*(q)$.
\end{definition}

It is evident then from (\ref{conorder}) that our pair of polar maps $(I^*,I_*)$ are a Galois connection, and our use of ``Galois pair'' notation $\gal{A}{B}$ for pairs $A \sub G, B \sub M$ evokes this sense. 

Categorically, Galois connections arise when considering posets as categories and looking for adjoint functors between them\cite{Bradley2024,FoBSpD18}. 
For our purposes we will not need this formality and instead just cite the following lemma for properties of such maps.

\begin{lemma} \cite[Lemma 7.26]{davey_priestley_2002}
  Assume $(f^*,f_*)$ is a Galois connection between ordered sets $P$ and $Q$.
  Let $p,p_1,p_2 \in P$ and $q,q_1,q_2 \in Q$.
  Then 
  \begin{itemize}
    \item[(Gal1)] $p \leq f_*(f^*(p))$ and $f_*(f^*(q)) \leq q$,
    \item[(Gal2)] $p_1 \leq p_2 \implies f^*(p_1) \leq f^*(p_2)$ and $q_1 \leq q_2 \implies f_*(q_1) \leq f_*(q_2)$, and
    \item[(Gal3)] $f^*(p) = f^*(f_*(f^*(p)))$ and $f_*(q) = f_*(f^*(f_*(q)))$.
  \end{itemize}
\end{lemma}

Viewing the polar maps as a Galois connection allows us to utilize the following result to see that iterating them does indeed form a closure operator on $\pwr(G)$ and $\pwr(M)$ respectively.

\begin{proposition} \cite[Remark 7.27]{davey_priestley_2002}
  \label{prop:gc_closure}
  Let $(f^*,f_*)$ be a Galois connection between ordered sets $P$ and $Q^{op}$.
  Then 
  \begin{enumerate}
    \item $c :=  f_*\circ f^* : P \to P$ and $k :=  f^*\circ f_* : Q \to Q$ are closure operators.
    \item Let
$$
        P_c := \{ p \in P : f_*(f^*(p)) = p\}
      \text{ and }
        Q_k := \{ q \in Q : f^*(f_*(q)) = q \}.
$$
      Then $f^* : P_c \to Q_k^{op}$ and $f_* : Q_k^{op} \to P_c$ are mutually inverse order isomorphisms.
  \end{enumerate}
\end{proposition}

The pairs of isomorphic closed sets under the Galois correspondence formed from the polar maps of a context are the formal concepts. Specifically, for every set of objects $A \sub G$, we know that $c(A)$ is the extent of a concept denoted 
    \[ \mathfrak{B}^*(A) = \gal{c(A)}{I^*(A)} = \gal{A''}{A'} \in \mathfrak{B}(G,M,I)   \] 
    as a Galois pair; and dually for every set of attributes $B \sub M$, $k(B)$ is the intent of the concept
\[ \mathfrak{B}_*(B) = \gal{I_*(B)}{k(B)} = \gal{B'}{B''} \in \mathfrak{B}(G,M,I),    \]
where we introduce the functions $\func{\mathfrak{B}^*}{2^G}{\mathfrak{B}(G,M,I})$ and $\func{\mathfrak{B}_*}{2^M}{\mathfrak{B}(G,M,I})$ taking object and attribute sets respectively to their formal concepts.

These closure operations are also evident in our example. Referring to \sec{eg:running}, Table \ref{tab:running}, and \fig{fig:concept_lattice},  Table \ref{tab:cona} shows some examples for object sets, and Table \ref{tab:conb} shows some examples for attribute sets (using compact set notation).

\begin{table}[]
    \centering
    \begin{tabular}{c||c|c|c}
         $A \sub G$ & $I^*(A) = A'$ & $I_*(I^*(A)) = c(A) = A''$ &  $\mathfrak{B}^*(A) = \gal{c(A)}{I^*(A)} = \gal{A''}{A'}$ \\
         \hline
$d$ & $012$ & $ad$ & $\gal{ad}{012}$  \\
$ad$ & $012$ & $ad$ & $\gal{ad}{012}$  \\
$ab$ & $5$ & $abc$ & $\gal{abc}{5}$   \\
    \end{tabular}
    \caption{Polar maps and concepts for some object sets, using compact set notation.}
    \label{tab:cona}
\end{table}

\begin{table}[]
    \centering
    \begin{tabular}{c||c|c|c}
         $B \sub M$ & $I_*(B) = B'$ & $I^*(I_*(B)) = k(B) = B''$ &  $\mathfrak{B}_*(B) = \gal{I_*(B)}{k(B)} = \gal{B'}{B''}$ \\
         \hline
$2$ & $ad$ & $012$ & $\gal{ad}{012}$  \\
$34$ & $b$ & $345$ & $\gal{b}{345}$  \\
$1$ & $acd$ & $1$ & $\gal{acd}{1}$   \\
    \end{tabular}
    \caption{Polar maps and concepts for some attribute sets, using compact set notation.}
    \label{tab:conb}
\end{table}



\section{Concept Lattices and Hypergraphs}
\label{sec:hypergraph_complexes}

\subsection{The Category of Formal Contexts}

Let us begin with the definition of $\cat{Ctx}$, 
which has the same objects as $\cat{Rel}$ and thus also the same as $\cat{MHyp}$. What is different is what we ask of the morphisms. We ask that they preserve concepts instead of just the incidence relation or sending hyperedges to hyperedges.

\begin{definition}
  \label{def:ctx}
  The objects of the category $\cat{Ctx}$ consist of formal contexts.
  Each object is a triple $(G,M,I)$ of two sets and a Boolean characteristic matrix, respectively.

  A morphism $m: (G_1,M_1,I_1) \to (G_2,M_2,I_2)$ consists of two functions $f: G_1 \to G_2$ and $g: M_1 \to M_2$ such that
  \begin{equation*}
    f(A)' = g(A') \text{ and } g(B)' = f(B'),
  \end{equation*}
  for all $A \subseteq G_1$ and all $B \subseteq M_1$,
  recalling that by $f(A)$ (etc.), we are referring to the image of the set $A$ under $f$.
  
  Composition of morphisms is defined by composing the respective pairs of functions on objects and attributes.
  Namely, if $m_1 : (G_1,M_1,I_1) \to (G_2,M_2,I_2)$ is given by $f_1: G_1 \to G_2$ and $g_1 : M_1 \to M_2$ and
  $m_2 : (G_2,M_2,I_2) \to (G_3,M_3,I_3)$ is given by $f_2: G_2 \to G_3$ and $g_2 : M_2 \to M_3$,
  then $(m_2 \circ m_1)$ is given by $(f_2 \circ f_1) : G_1 \to G_3$ and $(g_2 \circ g_1) : M_1 \to M_3$.
\end{definition}

\begin{corollary}
\label{cor:ctx_subcat_mhyp}
    $\cat{Ctx}$ is a non-full subcategory of $\cat{MHyp}$, since the objects are the same, and the morphisms of $\cat{Ctx}$ merely require a stricter condition to be satisfied.
\end{corollary}

In many situations, one is concerned with the objects \emph{only},
and in this case, formal contexts, binary relations, Boolean matrices, and multi hypergraphs are indistinguishable.
However, the appropriate transformations---the morphisms---in each of these settings are quite different.
For instance, if one considers Boolean matrices, one might wish to treat two matrices under similarity transformations as isomorphic. 
This isomorphism may be inappropriate for the pair of binary relations generated by these two matrices.
In our case, the morphisms of $\cat{Ctx}$ are chosen to connect formal contexts with their concepts. The authors will also note that the morphisms used in our category of contexts differ from those found in other literature \cite{Krtzsch2005}. 

\begin{lemma}
  \label{lem:ctx_morphisms}
  Morphisms of $\cat{Ctx}$ preserve concepts.
  Explicitly, if $(f,g) : (G_1,M_1,I_1) \to (G_2,M_2,I_2)$ is a morphism of $\cat{Ctx}$, then $f(A'') = f(A)''$ and $g(B'') = g(B)''$
  for all $A \subseteq G_1$ and all $B \subseteq M_1$.
\end{lemma}
\begin{proof}
  This is simply calculation.
  If $A'=B$, the definition of a $\cat{Ctx}$ means that $f(B') = g(B)'$, which implies $f(A'') = g(A')'$.
  The other condition on a morphism asserts that $f(A)' = g(A')$, which implies $f(A)'' = g(A')'$.
  Therefore $f(A'') = f(A)''$.
  The other statement follows \emph{mutatis mutandis}.
\end{proof}

We remind the reader that there is an inclusion of categories $\cat{Ctx} \to \cat{Rel}$.

\begin{proposition}
\label{prop:ctx_sub_rel}
  Suppose that $(f,g) : (G_1,M_1,I_1) \to (G_2,M_2,I_2)$ is a morphism of $\cat{Ctx}$.
  If $a \in G_1$ and $b \in M_1$ such that $a I_1 b$, then $f(a) I_2 g(b)$.
\end{proposition}

\begin{proof}
  Suppose $aI_1b$.  The intent of $a$ is $\{a\}' = \{m \in M_1 : a I_1 m\}$.
  We assumed that $\{a\}'$ is nonempty by hypothesis, since $b \in \{a\}'$.
  Hence $g(\{a\}')$ is nonempty, and contains at least $g(b)$.
  On the other hand, assuming that we have a $\cat{Ctx}$ morphism means that
  $f(a)' = \{ m \in M_2 : f(a) I_2 m \} = g(\{a\}').$
  Hence $f(a)$ is related to $g(b)$, or in other words $f(a) I_2 g(b)$.
\end{proof}

Recalling Definition \ref{def:hyprep}, the following establishes the relationship between formal contexts and hypergraphs.

\begin{proposition}
\label{prop:ctx_to_hyp}
  Each morphism in $\cat{Ctx}$ given by $(f,g) : (G_1,M_1,I_1) \to (G_2,M_2,I_2)$ induces a hypergraph morphism $HypRep(G_1,M_1,I_1) \to HypRep(G_2,M_2,I_2)$ simply by acting on the vertices.
  Therefore, $HypRep : \cat{Ctx} \to \cat{Hyp}$ is a covariant functor.
\end{proposition}
\begin{proof}
  Suppose that $e$ is an edge of $HypRep(G_1,M_1,I_1)$.
  This means that there is an $m\in M_1$ such $e= \{m\}'$.
  Since $(f,g)$ is a $\cat{Ctx}$ morphism, we have that $f(\{m\}') = g(m)'$.
  Therefore, $f(\{m\}')$ is in fact that hyperedge defined by $g(m)$ in $HypRep(G_2,M_2,I_2)$.
  
  Composition of $\cat{Ctx}$ morphisms involves composition of functions, and so does composition hypergraph morphisms.
  Both compositions occur in the same order, so we therefore have a covariant functor.
\end{proof}

\subsection{Concept Lattices} \label{sec:cl}

Assume a binary relation $R$ taken as a formal context $C=(G,M,I)$. Then for each pair of concepts $K_1=(A_1,B_1), K_2=(A_2,B_2) \in \setsys{B}(G,M,I)$, we define an order $K_1 \le K_2$ as $A_1 \subseteq A_2$. It then follows that $B_1 \supseteq B_2$. 

\begin{proposition} \cite[Proposition 3.6]{davey_priestley_2002} \label{prop:context_complete_lattice}
  Let $(G,M,I)$ be a context with concepts $K_j = (A_j,B_j) \in \cl, 1 \le j \le |\cl|$.
  Then $(\cl;\leq)$ is a complete lattice in which join and meet are given by the following:
  \begin{equation*}
    \bigvee_{K_j \in \cl(G,M,I)}(A_j,B_j) = \left(\left(\bigcup_{K_j \in \cl(G,M,I)}A_j\right)'',\bigcap_{K_j \in \cl(G,M,I)}B_j \right),
  \end{equation*}
  and 
  \begin{equation*}
    \bigwedge_{K_j \in \cl(G,M,I)}(A_j,B_j) = \left(\bigcap_{K_j \in \cl(G,M,I)}A_j,\left(\bigcup_{K_j \in \cl(G,M,I)}B_j\right)'' \right).
  \end{equation*}
\end{proposition}

We can hereby also call $\setsys{B}(G,M,I)$ the {\bf concept lattice} of the context $C$ with this order in mind. The concept lattice for our example has already been shown above in Figure \ref{fig:concept_lattice}, with extent $A \sub G$ and intent $B \sub M$  as a Galois pair. The ``full context lattice'' is on the left-hand side, and a reduced form on the right, which we now describe.


For every concept $\gal{A}{B} \in \cl$, its extent $A$ contains the unions of the extents below it, and its intent $B$ the unions of the intents above it:
    \[ A \supseteq \bigcup_{(A_*,B) < K} A_*, \qquad
         B \supseteq \bigcup_{(A,B^*) > K} B^*    \]
To cover the case when the above are not equality, for every concept $K=\gal{A}{B} \in \cl$ we can denote
    \[ A^- = A \setminus \left( \bigcup_{(A_*,B) < K} A_* \right) \sub G,   \qquad
        B^- = B \setminus \left( \bigcup_{(A,B^*) > K} B^* \right) \sub M,   \]
    and create a new {\bf reduced concept lattice} where each node $K \in \cl$ now shows the reduced pair $K^-=\gal{A^-}{B^-}$. 
    In turn, the full concept lattice can be reconstructed from the reduced as follows: for each node in the lattice, union the objects of the node with those seen below to form the extent, and union the attributes of the node and those seen above to form the intent.

The reduced concept lattice is shown on the right of \fig{fig:concept_lattice}. Note that it's common for many nodes in the reduced concept lattice to have empty intents, empty extents, or both. 
For nodes with empty extents (intents) in the reduced lattice, their extents (intents) in the full lattice can then be read off as the union of the extents (intents) of their children (parents). 
    In our example for $K=\gal{ad}{012}$, we have $K^-=\gal{d}{02}$, while for $K=\gal{abc}{5}$ we have $K^-=\gal{}{5}$, while  $K=\gal{ac}{15}$ is a ``fully blank'' node whose extent is an explicit union of its two children, and intent that of its two parents. 

\rem{

\begin{figure}
  \begin{center}
    \begin{tabular}{cc}
      \subfloat[The fully labeled concept lattice]{\includegraphics[width = 3in]{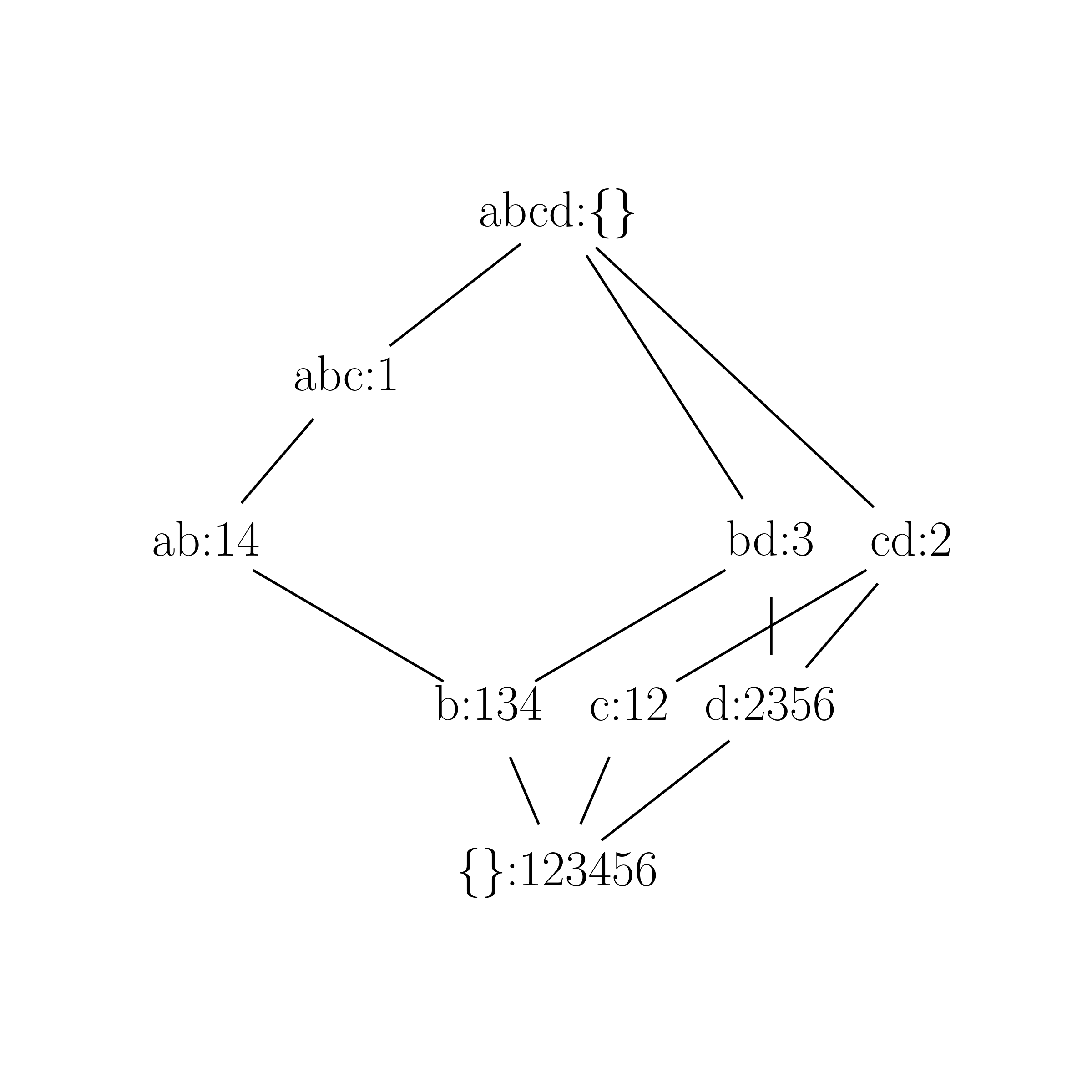}} &
      \subfloat[The reduced label concept lattice]{\includegraphics[width = 3in]{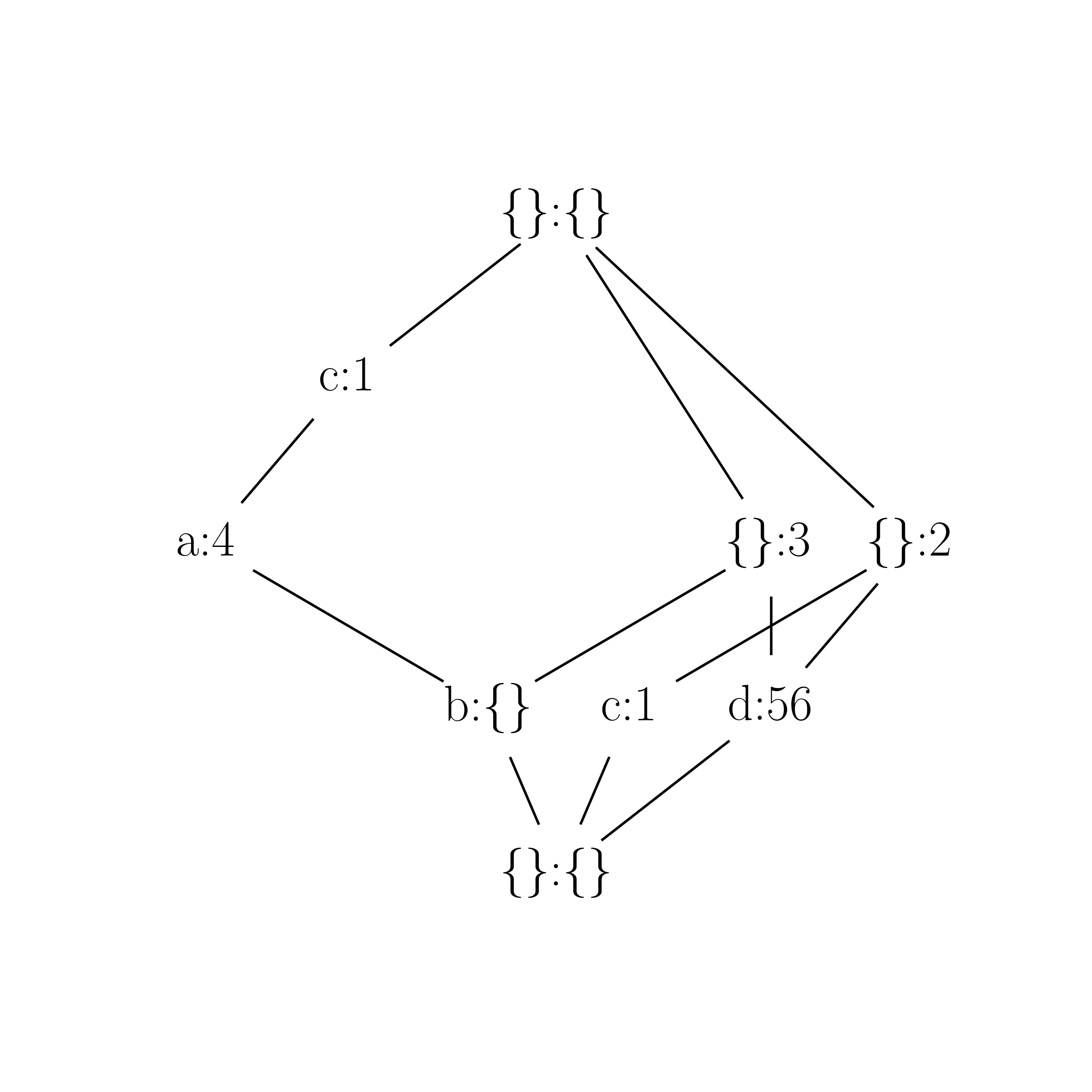}} 
    \end{tabular}
    \caption{Example concept lattices.}
    \label{fig:running_concept}
  \end{center}
\end{figure}

}

\begin{corollary}[To Proposition \ref{prop:context_complete_lattice}]
  For the context $(G,M,I)$, the set of extents $\setsys{B}_M$ and dually the set of intents $\setsys{B}_G$ are closed under intersection.
\end{corollary}

\begin{lemma} \cite[Theorem 3.8]{davey_priestley_2002}
  \label{lem:fund_1}
  Let $(G,M,I)$ be a context and $L = \setsys{B}(G,M,I)$ the associated complete lattice of concepts.
  Then the mappings $\gamma:G\to L$ where $\gamma(g) := ({g''},{g'})$ and $\mu : M \to L$ where $\mu(m) = ({m'},{m''})$ are such that the set $\gamma(G)$ is join-dense in $L$,
  the set $\mu(M)$ is meet-dense in $L$,
  and $gIm$ is equivalent to $\gamma(g) \leq \mu(m)$ for every $g \in G$ and $m \in M$.
\end{lemma}

\begin{lemma} \cite[Theorem 3.9]{davey_priestley_2002}
  \label{lem:fund_2}
  Let $L$ be a complete lattice and let $G$ and $M$ be sets for which there exist mappings $\gamma:G \rightarrow L$ and $\mu : M \rightarrow L$, such that $\gamma(G)$ is join-dense in $L$ and $\mu(M)$ is meet dense in $L$.
  Define $I$ by $gIm$ if and only if $\gamma(g) \leq \mu(m)$, for all $g \in G$ and $m \in M$.
  Then $L$ is isomorphic to $\setsys{B}(G,M,I)$.
  In particular, any complete lattice $L$ is isomorphic to the concept lattice $\setsys{B}(L,L,\leq)$.
\end{lemma}

When considered together, Lemmas \ref{lem:fund_1}--\ref{lem:fund_2} above are considered the ``fundamental theorem of concept lattices''.
They make explicit the relationship between contexts and concept lattices.

The reader is cautioned that Lemma \ref{lem:fund_2} \emph{does not establish} that this is an equivalence between categories $\cat{Ctx}$ and $\cat{CompLat}$.
Many contexts that are not isomorphic in $\cat{Ctx}$ may yield the same complete lattice.
Upon applying Lemma \ref{lem:fund_2} to render this complete lattice into a context will not recover all of these contexts!

For this reason, we will make a distinction between a concept lattice, as an object in $\cat{CompLat}$ for which any original context has been forgotten, versus a concept lattice itself which contains strictly more information.

\begin{definition}
\label{def:conlat}
A \textbf{concept lattice} consists of a tuple $(G,M,L,\gamma,\mu)$: a lattice $L \in \cat{CompLat}$, two sets $G$ and $M$, and two functions $\gamma$ and $\mu$ as defined in Lemma \ref{lem:fund_2}.  

A morphism between concept lattices $(G_1,M_1,L_1,\gamma_1,\mu_1) \to (G_2,M_2,L_2,\gamma_2,\mu_2)$ consists of a $\cat{Ctx}$ morphism between the contexts $(G_1,M_1,I_1) \to (G_2,M_2,I_2)$ defined by Lemma \ref{lem:fund_2}.
\end{definition}

Because the definition of $\cat{ConLat}$ effectively reuses the definition of morphisms from $\cat{Ctx}$, we have the following result.

\begin{corollary} 
\label{cor:context_concept}
  The operation of constructing a concept lattice $\setsys{B}(G,M,I)$ from its associated formal context $(G,M,I)$ is a covariant functor $\setsys{B} : \cat{Ctx} \to \cat{ConLat}$.
  Moreover, this is an isomorphism of categories.
\end{corollary}
\begin{proof}
    We need to show that morphisms in $\cat{Ctx}$ induce lattice homomorphisms on their associated concept lattices.

      Suppose that $m: (G_1,M_1,I_1) \to (G_2,M_2,I_2)$ is a morphism in $\cat{Ctx}$ given by $m_G: G_1 \to G_2$ and $m_M: M_1 \to M_2$.
  Let us define $m_L : \setsys{B}(G_1,M_1,I_1) \to \setsys{B}(G_2,M_2,I_2)$ as follows.
  Suppose that $K$ is a concept of $\setsys{B}(G_1,M_1,I_1)$.
  Without loss of generality, Lemma \ref{lem:fund_1} implies that we can write this as $K=(A,A')$.
  This means that for each $a \in A$ and each $b \in A'$, we have that $I_1[a,b]=1$.
  Since $m$ is a morphism in $\cat{Ctx}$, this means that $I_2[m_G(a),m_M(b)]=1$ as well.
  Therefore, $(m_G(A),m_M(A'))$ is at least a subconcept of $\setsys{B}(G_2,M_2,I_2)$.
  According to Lemma \ref{lem:fund_1}, we therefore may define
  \begin{equation*}
    \begin{aligned}
      m_L((A,A')) &:= (\gamma_2 \circ m_G)(A) = (m_G(A)'',m_G(A)') \\&= (\mu_2 \circ m_M)(A') = (m_M(A')',m_M(A')''),
    \end{aligned}
  \end{equation*}
  which is guaranteed to be a concept of $\setsys{B}(G_2,M_2,I_2)$.
\end{proof}


\begin{proposition}
\label{prop:concept_lattice_morphism}
    The action of forgetting the Galois data from a concept lattice $(G,M,L,\gamma,\mu)$ in $\cat{ConLat}$ to produce its underlying complete lattice $L$ in $\cat{CompLat}$ is a functor $\cat{ConLat} \to \cat{CompLat}$.
\end{proposition}
\begin{proof}
We must establish that $m_L$, defined in Corollary \ref{cor:context_concept} is a lattice homomorphism.
  Suppose that we have a collection of concepts $(A_i,A_i')$ of $\setsys{B}(G_1,M_1,I_1)$.
  We have that

  \begin{equation*}
    \begin{aligned}
      m_L\left(\bigvee_i (A_i, A_i') \right)
      &= m_L\left(\left(\bigcup_i A_i \right)'',\bigcap_i A_i'\right) \\ 
      = (\gamma_2 \circ m_G)\left(   \left(\bigcup_i A_i \right)''     \right) 
      &= \gamma_2 \left( m_G\left(   \left(\bigcup_i A_i \right)''     \right)\right)\\ 
      = \gamma_2 \left( m_G\left(   \left(\bigcup_i A_i \right)     \right)'' \right) 
      &= \gamma_2 \left(\bigcup_i m_G\left(   A_i      \right)'' \right)\\ 
      = \gamma_2 \left(\bigvee_i m_G\left(   A_i      \right) \right) 
      &= \bigvee_i \gamma_2 \left( m_G\left(   A_i      \right)\right)\\ 
      &= \bigvee_i m_L \left(   A_i      \right) 
    \end{aligned}
  \end{equation*}
  Notice the use of Lemma \ref{lem:ctx_morphisms} on the fourth line and Proposition \ref{prop:context_complete_lattice} on the sixth line.
  The argument for $\bigwedge_i (A_i,A_i')$ follows by the same reasoning, after changing $m_G$ into $m_M$ as appropriate.
\end{proof}

\rem{

\caj{I'm confused, how is the above related to the below?} \rawson{I think the definition could be moved, but we do use it later.}

\begin{definition}
  For a context $C=(G,M,I)$ define $C^T$ to be the context $(M,G,I^T)$ where $mI^Tg$ if and only if $gIm$.
\end{definition}

Let $\setsys{B}_G := \{A\subseteq G : A'' = A\}$ and $\setsys{B}^\delta_M := \{B\subseteq M : B'' = B\}$ where $\setsys{B}^{\delta}_{M}$ is the reverse inclusion order on $\setsys{B}_M$.
This forms an order isomorphism in the commutative diagram drawn in Figure \ref{fig:concept_lattice_duality}. \rawson{Proof?}
It is worth explicitly noting that the lattices formed over each of the three objects in this diagram are also isomorphic to each other. \rawson{Proof?}
In Figure \ref{fig:concept_lattice_duality} $\pi_1$ is the projection map $\pi_1 (A,B) \to A$ and $\pi_2$ is the projection map $(A,B) \to B$. \rawson{Proof?}
The functions $t$ and $t^{-1}$ establish correspondences between a concept $(A,B)$ and its image $(B,A)$ in the dual,
and are necessary to make the diagram commute. \rawson{This sentence doesn't make sense. The $t$ should be the $'$ function.}

\begin{figure}
  \begin{center}
    \begin{tikzcd}
      & \setsys{B}(G,M,I) \arrow[dl, "\pi_1"] \arrow[dr ,"\pi_2"] & \\
      \setsys{B}_G \arrow[shift right]{rr}[below]{t}
      & & \setsys{B}^{\delta}_{M} \arrow[shift right]{ll}[above]{t^{-1}}
    \end{tikzcd}
    \caption{The concept lattice and its dual are order isomorphic.}
    \label{fig:concept_lattice_duality}
  \end{center}
\end{figure}

}

\subsection{The Intersection Complexes of Hypergraphs}

\begin{definition}
  A set system $E \subseteq \mathcal{P}(V)$ is an \textbf{intersection complex} if $E$ is closed under multi-way intersection.
  Namely, for all $F \sub E$ then $\cap_{f \in F} f \in E$. It is said to be \textbf{topped} if $V \in E$.
  

\end{definition}

An intersection complex is also called an intersection structure, intersection closure, or $\cap$-structure. Given a hypergraph $H=(V,E)$, we can form its intersection complex $H^\cap = (V,E^\cap)$, by including all multi-way intersections: 
\begin{equation*}
  E^\cap = \left\{ \bigcap_{f\in F} f \right\}_{F \sub E, F \neq \emptyset}.
\end{equation*}

\rem{

Figure \ref{fig:original_hyper} shows the original hypergraph on the left, the collapsed hypergraph in the middle, and its intersection complex on the right. 

\begin{figure}
  \begin{center}
    \caption{Various hypergraphs obtained from the context in Example \ref{eg:running}. \rawson{Need to include the edge with no vertices in (c).}}
    \begin{tabular}{ccc}
      \subfloat[$H$]{\includegraphics[width = 2in]{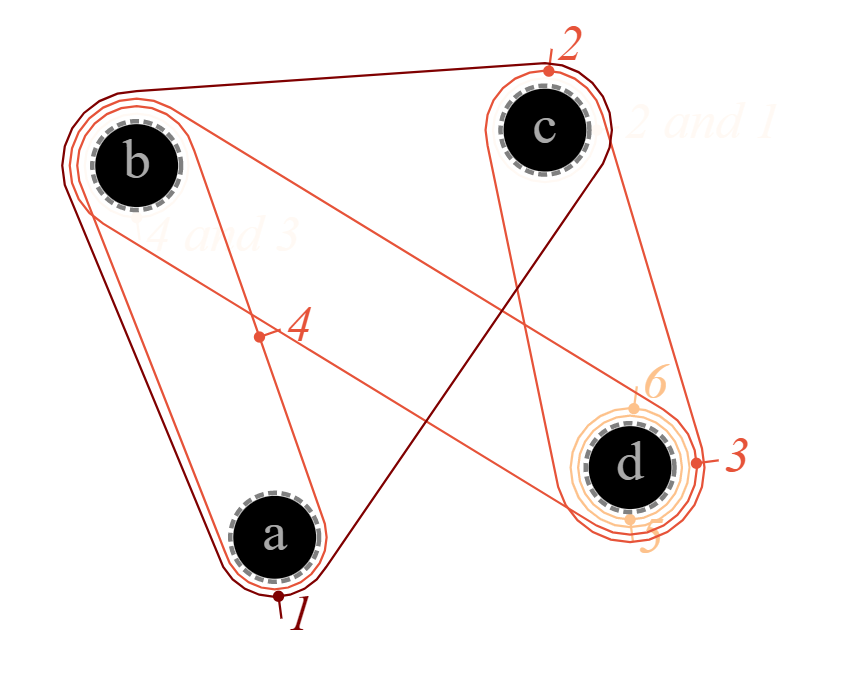}} &
      \subfloat[$H_c$]{\includegraphics[width = 2in]{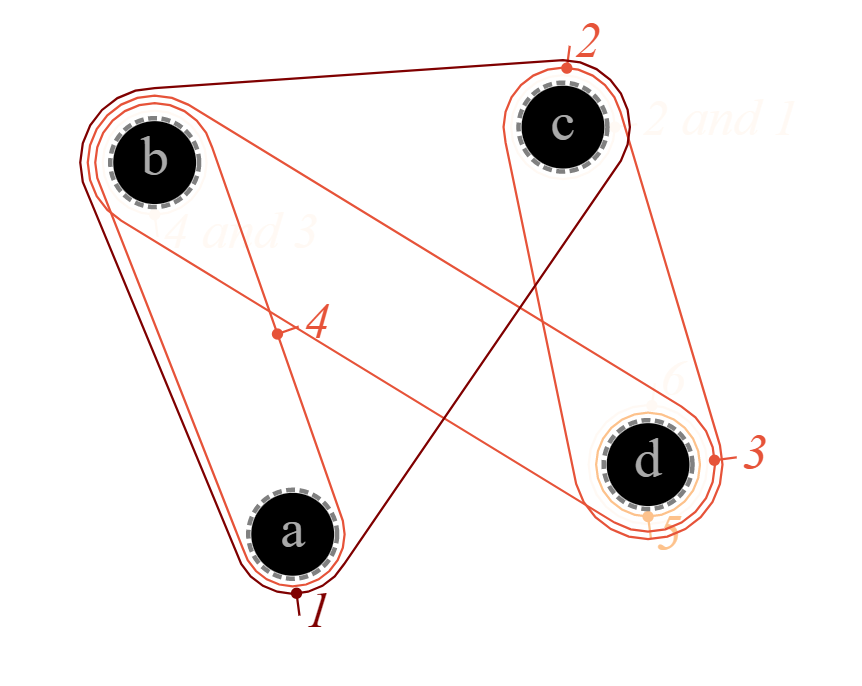}} &
      \subfloat[$H_c^\cap$]{\includegraphics[width = 2in]{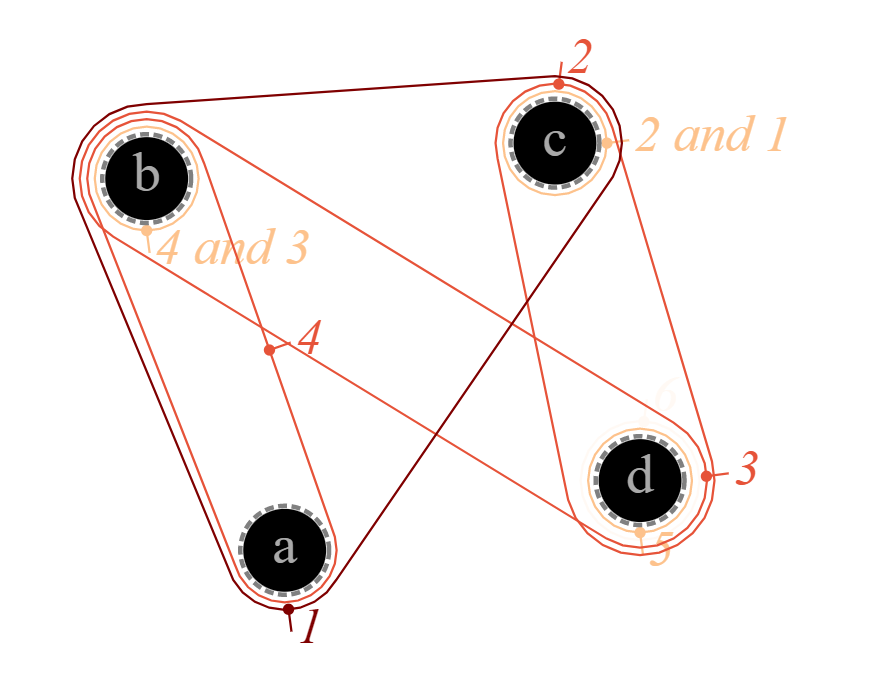}} 
    \end{tabular}
    \label{fig:original_hyper}
  \end{center}
\end{figure}

}

\fig{fig:int} shows the intersection complex $H^\cap$ for our example, with the Euler diagram on the left and the edge order on the right. It is created by adding the hyperedges (using compact set notation) $6=1 \cap 5 = a, 2 \cap 4 = c$, and $2 \cap 5 = ac$ to the original hypergraph $H$ as shown in Figures \ref{fig:hyper_ex} and \ref{figures/hg_hasse}. 
Compare this intersection complex $H^\cap$ to the ASC $H^\sub$ from \fig{asc}, which we can see is obtained simply by adding the further hyperedges $d,cd$, and $ab$. In fact, this relationship holds generally.

\begin{figure}[h]
  \begin{center}
    \begin{tabular}{cc}
      \subfloat{\includegraphics[width = 3.5in]{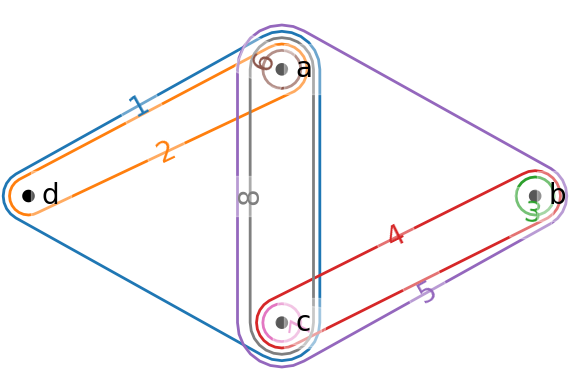}} & \qquad
      \subfloat{\includegraphics[width = 2in]{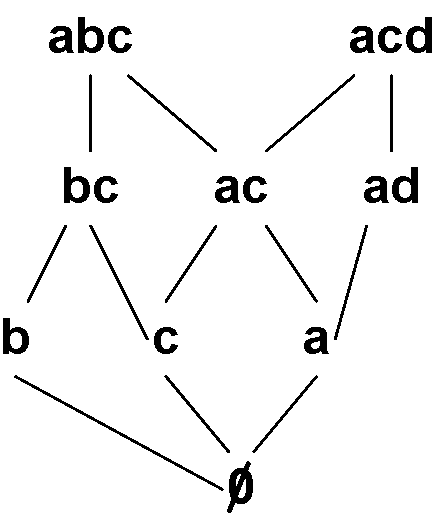}} 
      
    \end{tabular}
    \caption{(Left) Euler diagram of the intersection complex of our example. (Right) Its edge order $E^\cap$. }
    \label{fig:int}
  \end{center}
\end{figure}

\begin{proposition}
    
For  a hypergraph 
$H=(V,E)$, its edge order is included in its intersection complex, which is in turn included in its Dowker complex:
    \[ E \sub E^\cap \sub E^\sub. \]

\end{proposition}

This is illustrated in \fig{figures/objects3}. There is also a natural order theoretical representation of an intersection complex.

\begin{proposition} \cite[Corollary 2.32]{davey_priestley_2002} 
  \label{prop:intersection_lattice}
  Let $E$ be a topped intersection complex. Then its edge order $(E,\sub)$ is a complete lattice in which for every non-empty subset $F \sub E$,
  \begin{equation*}
    \bigwedge_{f \in F} f = \bigcap_{f \in F} f,	\qquad
    \bigvee_{f \in F} f = \bigcap\left\{ e \in E : \left( \bigcup_{f \in F} f \right) \sub e \right\}.
  \end{equation*}
\end{proposition}

We also get the following result which connects closure operators and topped intersection complexes.

\begin{proposition} \cite[Theorem 7.3]{davey_priestley_2002}
  \label{prop:intersection_closure_operator}
  Let $C$ be a closure operator on a set $X$.
  Then the family
  \begin{equation*}
    \setsys{L}_C := \{A \subseteq X : C(A) = A \}
  \end{equation*}
  of closed subsets of $X$ is a topped $\bigcap$-structure and so forms a complete lattice, when ordered by inclusion, in which
  \begin{equation*}
    \bigwedge_{i\in I} A_i = \bigcap_{i\in I} A_i
 \text{ and }
    \bigvee_{i\in I} A_i = C\left(\bigcup_{i\in I}A_i\right).
  \end{equation*}

  Conversely, given a topped $\bigcap$-structure $\setsys{L}$ on $X$, the formula
  \begin{equation*}
    C_\setsys{L}(A) := \bigcap\{B\in\setsys{L} : A\subseteq B \}
  \end{equation*}
  defines a closure operator $C_\setsys{L}$ on $X$.
\end{proposition}


\subsection{Intersection Complexes and Concept Lattices} \label{sec:complexes}

This section explores the relationship between the concept lattice and the lattice formed with the subset order on the intersection closure of the hypergraph formed from a context.
We show that the edge order of the intersection complex of hypergraph induced by a formal context is lattice isomorphic, after a trivial extension, to the concept lattice induced by the formal context.

\begin{definition}
    Let $(G,M,I)$ be a context.
    The hypergraph obtained from the context $HypRep(G,M,I)$ will be denoted $H_c(G,M,I) = (G,E_c)$. Here $H_c$ the same $H_c$ that one obtains from $collapse(G,M,I)$ when treating $(G,M,I)$ as an element of $\cat{MHyp}$.
    The topped hypergraph, $H^t_c(G,M,I)=(G,E_c^t)$, is the hypergraph with possibly an additional hyperedge $E_c^t = G \cup E_c$.
\end{definition}

By definition the edge set of the intersection closure of a hypergraph forms an intersection complex.
If the hyperedge that is the whole set is present then it is a topped intersection complex. 

There is an order embedding from the simple hypergraph obtained from a context into the set of extents. The lemma below establishes that each hyperedge in $HypRep(G,M,I)$ is in fact the extent of a concept.

\begin{lemma}
  \label{lem:double_subset}
  Suppose that $A \subset G$ for some context $(G,M,I)$.
  If $A$ is $B'$ for some $B$ then $A = A''$.
\end{lemma}
\begin{proof}
  If $A = B'$ for some $B$ then $B \subseteq A'$ then $A'' \subseteq B'$ since $B\subseteq A'$ so $A'' \subseteq A$ so $A = A''$.
\end{proof}

\rem{

\begin{figure}
  \begin{center}
    \caption{Intersection closure of the hypergraph derived from Example \ref{eg:running}.}
    \begin{tabular}{cc}
      \subfloat[Hasse diagram of the edge set of $H_c^{t\cap}$ with the inclusion order]{\includegraphics[height = 2.5in]{figures/hasa_diagram_of_topped_primal_intersection complex.PNG}} &
      \subfloat[Concept lattice of $(G,M,I)$]{\includegraphics[width = 3in]{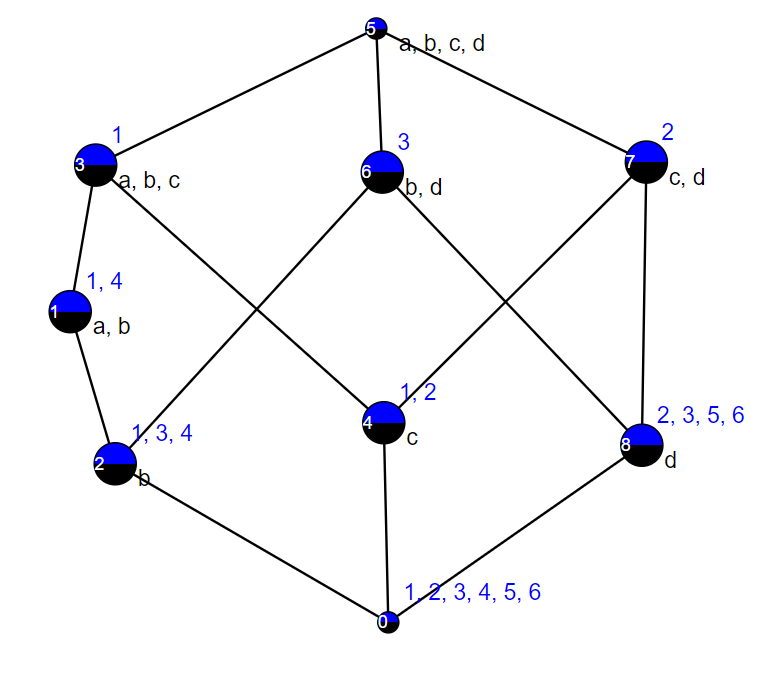}} 
    \end{tabular}
    \label{fig:hct_running}
  \end{center}
\end{figure}

\begin{example}
  Figure \ref{fig:hct_running} shows for the context in Example \ref{eg:running}, the poset associated to the intersection closure of the primal and the dual along with the concept lattice.
  It is arranged in such a way that the isomorphism between the two frames should be apparent.
\end{example}

}

\begin{lemma}
  We can interpret Proposition \ref{prop:intersection_lattice} as defining a functor $IntClose: \cat{Hyp} \to \cat{CompLat}$.
\end{lemma}
  One starts with a hypergraph, forms its intersection complex, adds a top if necessary, and then renders it into a complete lattice by constructing its edge order. 
\begin{proof}
  Suppose that $m : (V_1,E_1) \to (V_2,E_2)$ is a hypergraph morphism.
  That means that $m$ consists of a function $m : V_1 \to V_2$ such that $m(e_1) \in E_2$ for each $e_1 \in E_1$.
  Therefore, $m$ induces a function $E_1 \to E_2$, which acts on the elements of the edge orders $(E_1,\sub) \to (E_2, \sub)$.
  This function is not, in general, a lattice homomorphism. 
  Instead, for each $e \in E_1$, define
  \begin{equation*}
    \mu(e) := \bigcap_{\phi : e \subseteq \phi \in E_1} m(\phi).
  \end{equation*}
  Notice that each $m(\phi)$ is well defined, as is the intersection, because $E_2$ is an intersection complex.

  For an arbitrary subset $F \subseteq E_1$, we now have
  \begin{equation*}
    \begin{aligned}
      \mu\left(\bigwedge_{f\in F} f \right) &= \mu\left( \bigcap_{f\in F} f \right) \\
      = \smashoperator{\bigcap_{\phi : \left(\bigcap_{f\in F} f\right)  \subseteq \phi \in E_1}} m(\phi)
      &= \bigcap_{f\in F} \smashoperator[r]{\bigcap_{\phi :  f  \subseteq \phi \in E_1}} m(\phi) \\
      &= \bigwedge_{f\in F} \mu(f).
    \end{aligned}
  \end{equation*}
  The third line follows because $E_1$ is closed under intersection.
  Additionally, we have
  \begin{equation*}
    \begin{aligned}
      \mu\left(\bigvee_{f \in F} f\right) &= \mu\left( \bigcap\left\{ e \in E_1 : \left( \bigcup_{f \in F} f \right) \sub e \right\} \right)\\
      =\smashoperator{\bigcap_{\phi : \left(\bigcup_{f \in F} f \right) \subseteq \phi \in E_1}} m(\phi)
      &=\bigcap\left\{ e \in E_2 : \smashoperator{\bigcap_{\phi : \left(\bigcup_{f \in F} f \right) \subseteq \phi \in E_1}} m(\phi) \sub e \right\}\\
      &=\bigcap\left\{ e \in E_2 : \mu \left(\bigcup_{f \in F} f \right) \sub e \right\}\\
      &=\bigcap\left\{ e \in E_2 : \left( \bigcup_{f \in F} \mu(f) \right) \sub e \right\}\\
      &=\bigvee_{f \in F} \mu(f).
    \end{aligned}
  \end{equation*}
\end{proof}

\begin{theorem} 
  \label{thm:concept_iso_intclose}
  For a context $(G,M,I)$, the intersection closure of the topped hypergraph obtained from a context is lattice isomorphic to $\mathfrak{B}_G$.
  Briefly, we have the following isomorphism in $\cat{CompLat}$
  \begin{equation*}
  (IntClose \circ HypRep)(G,M,I) \cong \setsys{B}(G,M,I).
  \end{equation*}
\end{theorem}
The dual isomorphism also follows easily.
\begin{proof}

  By considering example 7.22 and remark 7.29 from \cite{davey_priestley_2002} we obtain that for the context $(G,M,I)$, the polar maps (Definition \ref{def:polar}) provide a closure operator on $\pwr(G)$. 
  We also know the set of closed sets under a closure operator form a topped intersection complex by Prop.\ \ref{prop:intersection_closure_operator}.

  For every $e\in E_c(G,M,I)$ there is, by definition, an $m\in M$ such that $e = {m}'$ so, by Lemma \ref{lem:double_subset}, $e$ is closed under $C$ in $\pwr(G)$.
  Each subset of $G$ which is closed is an extent by definition.
  Thus, since every edge is closed and the sets closed under a closure operator are an intersection complex, $H^{t\cap}_c(G,M,I)$ is a suborder of $\setsys{B}_G$ with the inclusion map.

  The only thing left to show is that $\setsys{B}_G \subseteq E^{t\cap}_c(G,M,I)$.
  If $A\in \setsys{B}_G$ then $A = B'$ for some $B$.
It follows directly from the definition of the polar map that  for $B\subseteq M$,
  \begin{equation}
  \label{eq:concept_iso_intclose}
B' =  \bigcap_{m\in B}m'.
  \end{equation}

  We again use the fact that by the definition of $E_c(G,M,I)$ that if $e\in E_c(G,M,I)$ then $e = m'$ for some $m\in M$.
  This combined with Equation \eqref{eq:concept_iso_intclose} above gives us that the set of all $B'$ is the set $\bigcap E^t_c(G,M,I)$.
  This set is closed under intersection, and is exactly $E^{t\cap}_c(G,M,I)$.
  In other words, we have shown that $\setsys{B}_G \subseteq E^{t\cap}_c(G,M,I)$.
\end{proof}


\section{Contexts and Topology}
\label{sec:concept_cosheaves}
According to the fundamental theorem of concept lattices (Lemmas \ref{lem:fund_1}--\ref{lem:fund_2} and Corollary \ref{cor:context_concept}), a formal context and a concept lattice determine each other completely.
On the other hand, the process of collapsing, which removes duplicate rows from the incidence matrix of a context, yields an irretrievable loss of information (Theorem \ref{thm:concept_iso_intclose}).
In exactly the same way, hypergraphs induced from a context $(G,M,I)$ using $HypRep$ also lose information about the identity and multiplicity of duplicated vertices (rows of $I$) and hyperedges (columns of $I$).
Since the lost information is local to the hyperedges,
we can capture this information by attributing it back to the duplicated hyperedges, which effectively ``undoes'' the collapsing process. 
The pertinent mathematical object for representing this attributed data is a \emph{cosheaf}.

\subsection{Cosheaf structures on set systems}
\label{sec:background_cosheaf}

Cosheaves allow one to specify data at various places within a partial order,
much like concept lattices.
In \cite{Robinson_relations}, it was recognized that the Dowker complex can be interpreted as constructing a cosheaf of abstract simplicial complexes on an abstract simplicial complex.
This Section combines these two insights to relate the cosheaf structure and concept lattice for a given context.

\begin{definition}\cite[Def. 12]{Robinson_relations}
    \label{def:coshvasc}
    The class of cosheaves of abstract simplicial complexes on abstract simplicial complexes is the class of objects of the category $\cat{CoShvAsc}$.  
    That is, each object of $\cat{CoShvAsc}$ is a cosheaf whose base space is an abstract simplicial complex, and whose spaces of local cosections are also abstract simplicial complexes.
    
    The morphisms of $\cat{CoShvAsc}$ are cosheaf morphisms subject to the constraint that both the base space and fiber maps are simplicial maps.
\end{definition}

In \cite[Thm. 4]{Robinson_relations}, it was shown that there is a functor $CoShvRep:\cat{Rel} \to \cat{CoShvAsc}$.  
When restricted to \textbf{total contexts}, those whose characteristic matrices have no zero rows and no zero columns, this becomes a faithful functor.
Using the isomorphism $\cat{ConLat}\to\cat{Ctx}$ and Proposition \ref{prop:ctx_sub_rel}, this means that there is a functor $\cat{ConLat}\to\cat{CoShvAsc}$ that transforms concept lattices into cosheaves.
This section proves in Theorem \ref{thm:concept_dowker} that this functor is injective on objects. 
Specifically, the cosheaves created by two distinct concept lattices are distinct themselves.

\begin{definition}
  \label{def:cosheaf}
  A \textbf{cosheaf of sets} $\cshf{C}$ \textbf{on a partial order} $(P,\leq)$ consists of the following specifications: 
  \begin{itemize}
  \item For each set $x \in P$ there is a set $\cshf{C}(x)$, called a \textbf{costalk} at x.
  \item For each $x \leq y \in P$,
    a function $(\cshf{C}(x \leq y)) : \cshf{C}(y) \to \cshf{C}(x)$,
    called the \textbf{extension} along $x \leq y$ such that whenever $x\leq y \leq z \in P$, $\cshf{C}(x\leq z) = ((\cshf{C}(x \leq y)) \circ  ((\cshf{C}(y \leq z))$.
  \end{itemize}
\end{definition}

The usual definition of a \textbf{cosheaf on a topological space} requires more axioms than appear in Definition \ref{def:cosheaf}.
In particular, a \textbf{cogluing} axiom is essential.  
The cosheaves used in this paper are always written over the Alexandrov topological space generated by the upward closed sets in the poset.
As such, what appears in Definition \ref{def:cosheaf} is the \emph{minimal specification} for a cosheaf on the Alexandrov topological space.

\begin{lemma}\cite[Cor. 2.5.21]{Curry}
  A cosheaf on a poset, as in Definition \ref{def:cosheaf}, uniquely specifies a cosheaf on the \textbf{Alexandrov topological space},
  for which the open sets are unions of upward closed sets in the poset.
\end{lemma}

Since the set of open sets in a topological space ordered by inclusion is a distributive lattice via Birkhoff's representation theorem (see e.g. \cite[pp.\ 252, 353]{StR97} or \cite{priestley_1970}), we have the following.
\begin{corollary} \cite{Robinson_relations}
  \label{cor:base_space_alexandrov}
  Given a cosheaf on a poset, the diagram of the corresponding cosheaf on the Alexandrov topological space is that of a distributive lattice.
  As such, the diagram has a unique top and bottom element.  
  This defines a covariant functor $Base : \cat{CoShvAsc} \to \cat{Asc}$.
\end{corollary}

Notationally, $\cshf{C}$ is used throughout the paper to define a general cosheaf on a poset,
whereas $\cshf{R}^0$ and $\cshf{R}$ are used to refer to the specific cosheaves defined in Definition \ref{def:dowker_cosheaf}.

\begin{definition}\cite{Robinson_relations}
  \label{def:dowker_cosheaf}
  Define a \textbf{cosheaf representation of a context} $(G,M,I)$ by the cosheaf $\cshf{R} = CoShvRep(G,M,I)$ on an abstract simplicial complex, defined by the following recipe.
  \begin{description}
  \item[Costalks:] If $\sigma$ is a simplex of the Dowker complex $(D \circ HypRep)(G,M,I)$ then 
    \begin{equation*}
      \cshf{R}(\sigma) := D(\sigma', \sigma, (I^T|_{\sigma,\sigma'})). 
    \end{equation*}
    (Recall that the prime operator was defined in Equation \eqref{eq:prime_def}).
    
  \item[Extensions:] If $\sigma \subseteq \tau$ are two simplices of the Dowker complex $D(G,M,I)$,
    then the extension $\cshf{R}(\sigma \subseteq \tau): \cshf{R}(\tau) \to \cshf{R}(\sigma)$ is the simplicial map along the inclusion $\tau' \hookrightarrow \sigma'$.
  \end{description}
  For brevity, we will often refer to $\cshf{R}$ as the \textbf{Dowker cosheaf}.

  The cosheaf arising from restricting costalks of $\cshf{R}$ to contain only the vertices will be denoted $\cshf{R}^0$.  Specifically,
  \begin{equation*}
  \cshf{R}^0(\sigma) = \sigma'.
  \end{equation*}
\end{definition}

  Figure \ref{fig:running_cosheaf} shows the Dowker cosheaf for the same context as in the running example.
  Notice that the arrows point downward.
  Since it is the diagram of a cosheaf, the corestriction maps are contravariant to the face inclusions.
  It is also worth noting that all the corestriction maps shown are inclusions. 

  \rem{

  Figure \ref{fig:running_dual_cosheaf} shows the Dowker cosheaf associated to the dual context.
  Notice how the Galois pairs of base and costalk simplices that appear in both Figures \ref{fig:running_cosheaf} and \ref{fig:running_dual_cosheaf} correspond exactly to the concepts in the concept lattice.

}

\begin{figure}[h]
  \begin{center}
\includegraphics[height=2in]{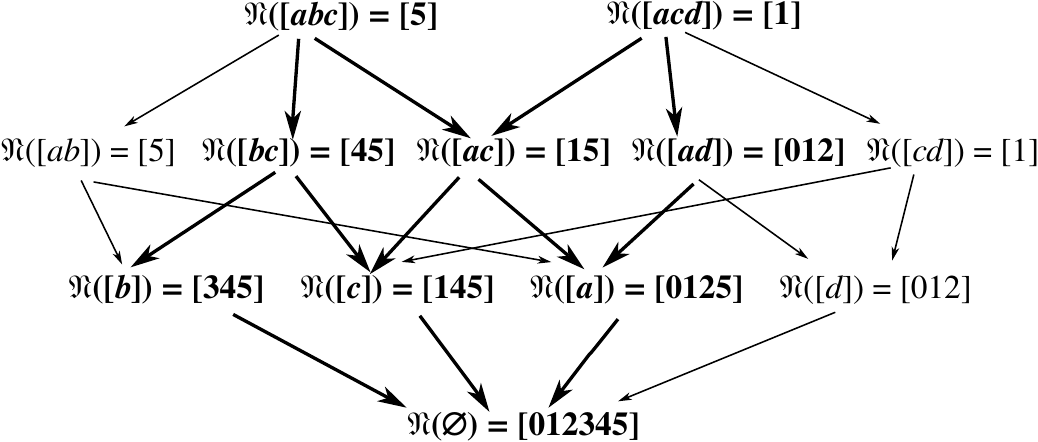}
    \caption{The Dowker cosheaf $\cshf{R}(G,M,I)$ for the context defined in Example \ref{eg:running}.  Concepts are marked in {\bf bold}. }
    \label{fig:running_cosheaf} 
  \end{center}
\end{figure}

Figure \ref{figures/cosheaf} shows the example using our Galois notation for simplicity. Comparing with Figure \ref{figures/H_final_CL2}, we can see the inclusion of the new Galois pairs $\gal{ab}{5}, \gal{cd}{1}$ and $\gal{d}{012}$, reflecting members of the Dowker cosheaf not part of the concept lattice.
In Figure \ref{fig:running_cosheaf}, these are the only non-bold cells, that is, the only cells that are not concepts.

\mypng{.4}{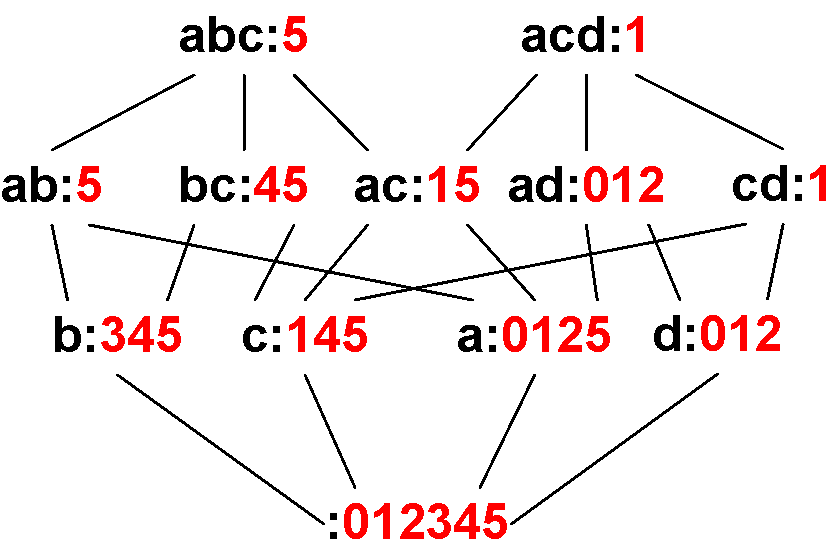}{Cosheaf in Galois notation.}

\subsection{Concept Lattices and Cosheaves}

This section elucidates the relationship between the concept lattice and the Dowker cosheaf.
Although we generally prefer to work with the Dowker cosheaf on the face poset of the Dowker complex to reduce notational clutter,
this clouds its connection to the concept lattice.
Specifically, the concept lattice is a labeled \emph{lattice}, but the face poset of the Dowker complex may not be a lattice.
To ensure the correspondence between the two objects,
we must work over the Alexandrov topology for the face poset.
This succeeds because of Corollary \ref{cor:base_space_alexandrov}, namely that the base space of the Dowker cosheaf on the Alexandrov space corresponds to a distributive lattice.

\begin{proposition}
  \label{prop:cl_embeds_dowker}
  There is an order embedding from $\setsys{B_G}(G,M,I)$ into the topped face poset $A$ of $D(G,M,I)$.
\end{proposition}
Recall that the topped face poset contains all of the simplices of $D(G,M,I)$ as well as the set of all vertices, which happens to be $G$.
\begin{proof}
  The elements of $\setsys{B}_G$ are of the form $A \subseteq G$ such that $A = A''$.
  
  The elements of $A$ are defined as $\sigma \in D(G,M,I)$ along with the whole set.
  Recall $\sigma \in D(G,M,I)$ if $\sigma' \neq \emptyset$.
  So $f: \setsys{B}_G \to A$ where $f$ is the inclusion map defines an order embedding since the order on both sets is the subset order.
\end{proof}

Notice that since the poset $A$ is a subposet of the open sets of the Alexandrov topology for $D(G,M,I)$,
Proposition \ref{prop:cl_embeds_dowker} can be interpreted as relating the concept lattice to the Dowker cosheaf of $(G,M,I)$.
This intuitive idea goes a bit further.

\begin{proposition}
  \label{prop:dowker_concepts_1}
  The set of unique costalks of the Dowker cosheaf $\cshf{R}$ for a context $(G,M,I)$ with the inclusion partial order forms a lattice.
  This partial order is isomorphic to $\setsys{B}_M$.
\end{proposition}

So while the lattice structure of $\setsys{B}_M$ can be recovered from the Dowker cosheaf, labeling the lattice with the actual concepts themselves allows complete recovery of the Dowker cosheaf.

\begin{proof}
  Recall for $\sigma \in D(G,M,I)$, we defined $\cshf{R}(\sigma) = D(\sigma', \sigma, (I^T|_{\sigma,\sigma'}))$.
  $D(\sigma', \sigma, (I^T|_{\sigma,\sigma'}))$ is just the minimal abstract simplicial complex such that $\sigma'$ is a face.
  So the set of unique costalks of $\cshf{R}$ is the set of abstract simplical complexes that are minimal for each $\sigma'$ for all $\sigma \in D(G,M,I)$.
  
  The face relation provides us with a partial order on these abstract simplicial complexes which is an order isomorphism to the set of unique $\sigma'$.
  Using Proposition \ref{prop:concept_inc_posrep}, the set of unique $\sigma'$ is exactly the set $\setsys{B}_M$.
  Therefore, there is naturally a lattice isomorphism from the set of unique costalks to $\setsys{B}_M$.
\end{proof}

\begin{corollary}
  Proposition \ref{prop:dowker_concepts_1} can be easily dualized to work on the Dowker cosheaf for the dual context.
\end{corollary}

\begin{proposition} (Compare with \cite[Props. 4, 5]{Robinson_relations})
  There is a covariant functor $PosRep: \cat{Cxt} \to \cat{Pos}$, which can be called the \textbf{poset representation of a context} that takes each $(G,M,I)$ to a collection $PosRep(G,M,I)$ of subsets of $\pwr(G)$,
  for which $A\in PosRep(G,M,I)$ if there is a $m \in M$ such that $gIm$ for every $g \in A$.
  The elements of $Posrep(G,M,I)$ are ordered by subset inclusion.
  This poset is isomorphic to the face poset on the Dowker complex of $(G,M,I)$.
\end{proposition}
\begin{proof}
The statement follows immediately from Theorem \ref{thm:concept_iso_intclose} and Proposition \ref{prop:intersection_lattice}.
\end{proof}

It is of note that the condition for $A \subseteq G$ to be in $PosRep(G,M,I)$ is equivalent to saying $A'$ is non empty.
This is weaker than the condition for $A \subseteq G$ to be in $\setsys{B}_G$ which is $A = A''$.
This is shown in Proposition \ref{prop:concept_inc_posrep} below.

\begin{proposition}
  \label{prop:concept_inc_posrep}
  Given a context $(G,M,I)$, the concept lattice lies within the image of its poset representation.
  Briefly, $\setsys{B}_G \subseteq PosRep(G,M,I) \cup \{G,\emptyset\} \subseteq \pwr(G)$.
  Moreover, the dual statement, $\setsys{B}_M \subseteq PosRep(M,G,I^T) \subseteq \pwr(M)$ is also true.
\end{proposition}
\begin{proof}
  By definition $A \subseteq G$ implies that $A \in \pwr(G)$.
  By definition for $A \subseteq G$, it follows that $A \in PosRep(G,M,I)$ if for every $g \in A$ there exists $m \in M$ such that $gIm$.
  This shows that $PosRep(G,M,I) \subseteq \pwr(G)$.
  For every $A \subseteq G$, it follows that $A \in \setsys{B}_G$ if $A'' = A$.
  (This is strictly stronger than the condition for $A$ to be in $PosRep(G,M,I)$.)
  This implies $\setsys{B}_G \subseteq PosRep(G,M,I)$. 
\end{proof}

Moreover, combining Propositions \ref{prop:cl_embeds_dowker} through \ref{prop:concept_inc_posrep} with the observation from \cite{Robinson_relations} that the Dowker cosheaf is an isomorphism invariant for contexts, we obtain the following pleasing result.

\begin{theorem}
  \label{thm:concept_dowker}
  The concept lattice is an invariant of the Dowker cosheaf of a context.  If two Dowker cosheaves for two contexts are isomorphic, then their concept lattices will be lattice isomorphic.  Conversely, the Dowker cosheaf itself can be recovered from a concept lattice.
  This establishes that the functor $\cat{ConLat} \to \cat{CoShvAsc}$ is injective on objects.
\end{theorem}



\section{(Co)homology of Dowker (Co)sheaves}
\label{sec:dowker_homology}

This section characterizes the cohomology of the sheaf that is dual to the Dowker cosheaf $\cshf{R}$.
In section \ref{sec:dowker_chain_complex}, we proceed to derive a stronger context invariant using a cosheaf of chain complexes derived from $\cshf{R}$.

In this section, we will work with chain complexes of $\mathbb{R}$-vector spaces via formal sums.
Although this greatly simplifies matters,
it is not obvious whether the results generalize to chain complexes with other types of coefficients.

The reader is reminded that chain complexes have their own category.

\begin{definition}\cite{weibel}
\label{def:kom}
A \textbf{chain complex} consists of an indexed collection of vector spaces $C_k$ and linear maps $\partial_k : C_k \to C_{k-1}$ such that $\partial_{k-1} \circ \partial_k = 0$ for all integer $k$.  

A \textbf{chain map} between one chain complex $(C_\bullet,\partial_\bullet)$ to another $(C'_\bullet,\partial'_\bullet)$ consists of an indexed collection of maps $m_k: C_k \to C'_k$ such that the diagram
\begin{equation*}
\xymatrix{
\dotsb \ar[r] & C_k \ar[r]^{\partial_k} \ar[d]^{m_k} & C_{k-1} \ar[r] \ar[d]^{m_{k-1}} & \dotsb \\
\dotsb \ar[r] & C'_k \ar[r]_{\partial'_k} & C'_{k-1} \ar[r] & \dotsb
}
\end{equation*}
commutes.

The category $\cat{Kom}$ has chain complexes for its objects and chain maps for its morphisms.
\end{definition}

Separately, it is useful to have notation that generalizes intent of a concept to handle sets which are not actually concepts.

\begin{definition}
  \label{def:M_sig}
  Suppose that $(G,M,I)$ is a context, and $\sigma \subseteq G$ be an arbitrary simplex of $D(G,M,I)$.
  (Recall that this means that there is an $m \in M$ such that if $g \in \sigma$, we have $gIm$.)
  Let
  \begin{equation*}
    \widehat{M_{\sigma}} := \{m\in M : gIm\text{ if and only if }g\in\sigma\}.
  \end{equation*}
\end{definition}

Notice the difference between $\widehat{M_{\sigma}}$ and
\begin{equation*}
    \sigma' = \{m \in M : gIm \text{ for all }g\in\sigma\}.
\end{equation*}
It follows that $\widehat{M_{\sigma}} \subseteq \sigma'$, but not necessarily conversely.

\begin{corollary}
  \label{cor:M_sig_hat}
  Every element of $M$ is an element of precisely one element in $\{\widehat{M_\sigma}\}_\sigma$.
\end{corollary}

\begin{corollary}
  If $(A,B)$ is a concept of $(G,M,I)$ then $\widehat{M_A} = B$.
\end{corollary}

\subsection{Dowker sheaf cohomology}

Because it has somewhat easier algebraic structure, it is a bit easier to start by computing the cohomology of the Dowker sheaf for the context instead of the homology of the Dowker cosheaf.  We return to the cosheaf in Section \ref{sec:dowker_chain_complex}.

\begin{definition}\cite{Robinson_relations}
  The \textbf{Dowker sheaf} $\shf{S}=ShvRep(G,M,I)$ on the Dowker complex $D(G,M,I)$ for a context $(G,M,I)$ is given by specifying its stalks and restrictions.
  For each $\sigma \in D(G,M,I)$, the stalk on $\sigma$ is
  \begin{equation*}
    \shf{S}(\sigma) := \vspan \sigma',
  \end{equation*}
  where we interpret columns of $I$ as basis elements and $\vspan S$ constructs the abstract vector space spanned by a set $S$.

  Notice that basis elements in each stalk can be compared between stalks since they correspond to elements of $M$.
  Moreover, if $\sigma \subseteq \tau$ then $\tau' \subseteq \sigma'$.
  We therefore define the restriction maps $\shf{S}(\sigma \subseteq \tau)$ between simplices to be given by vector space projections using the correspondence between identical columns appearing in the costalks of $\sigma$ and $\tau$.
\end{definition}

The main result of this section, Proposition \ref{prop:r0_h} below,
is that $\shf{S}$ has a simple cohomology when $(G,M,I)$ is a \textbf{total} context,
namely one in which every $g \in G$ is $I$-related to at least one $m \in M$,
and \emph{vice versa}.
The subcategory of $\cat{Ctx}$ consisting of total contexts is denoted $\cat{Ctx}_+$.

\begin{proposition}
  \label{prop:r0_h}
  Consider the Dowker sheaf $\shf{S}$ on the Dowker complex $D(G,M,I)$ constructed as the sheaf representation for a total context $(G,M,I) \in \cat{Ctx}_+$.
  The space of global sections of $\shf{S}$ is spanned by the elements of $M$ and decomposes according to the simplices of $D(G,M,I)$,
  \begin{equation*}
    H^0(\shf{S}) \cong \smashoperator{\bigoplus_{\sigma \in D(G,M,I)}} \vspan \{\widehat{M_\sigma}\}.
  \end{equation*}
  Moreover, $H^k(\shf{S})=0$ for all $k>0$.
\end{proposition}

We prove Proposition \ref{prop:r0_h} step-by-step through a sequence of Lemmas.

\begin{lemma}
  \label{lem:cech}
  Suppose $\shf{F}$ is a sheaf of vector spaces on a cell complex $X$,
  and that $\shf{F}'$ is a sheaf of vector spaces on the Alexandrov space generated by the face poset of $X$.
  Assume further that $\shf{F}$ and $\shf{F}'$ agree on every cell and have the same restriction maps between cells.
  Then the cellular sheaf cohomology of $\shf{F}$ is isomorphic to the usual (\v{C}ech) sheaf cohomology of $\shf{F}'$.  
\end{lemma}
\begin{proof}
  This is an easy application of the \v{C}ech ``good cover'' theorem for sheaves \cite{Bredon,Hubbard_2006}.  
\end{proof}

As a consequence of Lemma \ref{lem:cech}, we may use simplicial (cellular) sheaf cohomology to compute the sheaf cohomology of the Dowker sheaf.

\begin{lemma}
  \label{lem:r0_h0}
  Consider the Dowker sheaf $\shf{S}$ on the Dowker complex $D(G,M,I)$ for the context $(G,M,I) \in \cat{Ctx}_+$.
  The zeroth cohomology $H^0(\shf{S})$ is of dimension equal to the number of attributes $\#M$ of the context.
\end{lemma}
\begin{proof}
  The Dowker sheaf is constructed so that each stalk is given by
  \begin{equation*}
    \shf{S}(\sigma) := \bigoplus_{m \in \sigma'} \mathbb{R},
  \end{equation*}
  namely a vector space with dimension $\#\sigma'$, and whose basis can be taken to be $\sigma'$.
  Each restriction map $\shf{S}(\sigma \subseteq \tau)$ is simply the projection that is dual to the linear transformation induced by the inclusion $\tau' \to \sigma'$,
  which acts on the basis elements.
  That simply means that if we reinterpret $m \in \sigma'$ as a basis element of $\shf{S}(\sigma)$, then
  \begin{equation*}
    \left(\shf{S}(\sigma \subseteq \tau) \right)(m) = \begin{cases}
      m & \text{if }m \in \tau'\\
      0 &\text{otherwise.}
      \end{cases}
  \end{equation*}
  
  According to the simplicial sheaf cochain complex, 
  \begin{equation*}
    C^0(\shf{S}) = \bigoplus_{g\in G} \shf{S}([g]) = \bigoplus_{g\in G} \smashoperator[r]{\bigoplus_{m\in [g]'}} \mathbb{R},
  \end{equation*}
  and
  \begin{equation*}
    C^1(\shf{S}) = \bigoplus_{g_1\in G}\bigoplus_{g_2\in G} \shf{S}([g_1,g_2]) = \bigoplus_{g_1\in G} \bigoplus_{g_2\in G} \smashoperator[r]{\bigoplus_{m\in [g_1,g_2]'}} \mathbb{R}.
  \end{equation*}
  That is, the cochain space for the vertex $[g_1]$ contains a copy of $\mathbb{R}$ for each element of $[g_2]'$,
  and likewise for every edge $[g_1,g_2]$.
  Essentially, $C^0(\shf{S})$ carries at least one copy of each element of $M$ (assuming the context is an element of $\cat{Ctx}_+$),
  and the coboundary map $d^0 : C^0(\shf{S}) \to C^1(\shf{S})$ identifies duplicates.
  The kernel of $d^0$ is therefore indexed by the same set of elements of $Y$,
  but with the duplicates removed.
  To see this, recall that for $a \in C^0(\shf{S})$,
  $d^0(a)$ is given by the classic formula
  \begin{equation*}
    (d^0 (a))[g_1,g_2] = \shf{S}([g_2]\subseteq[g_1,g_2])(a([g_2])) - \shf{S}([g_1]\subseteq[g_1,g_2])(a([g_1])).
  \end{equation*}
  This vanishes precisely when $a([g_1])$ and $a([g_2])$ agree on all of the elements of $[g_1]' \cap [g_2]'$,
  so the kernel of $d^0$ restricted to this subspace of $C^0(\shf{S})$ is indexed by $[g_1]' \cup [g_2]'$.
  It is of dimension
  \begin{equation*}
    \#\left([g_1]' \cup [g_2]'\right) = \#[g_1]' + \#[g_2]' - \#\left([g_1]' \cap [g_2]'\right).
  \end{equation*}
  Repeating this process over all other edges completes the argument.
\end{proof}

\begin{lemma}
  \label{lem:r0_h0_splitting}
  The space of global sections of $\shf{S}=ShvRep(G,M,I)$ decomposes naturally into a direct sum of vector spaces, indexed by the simplices of $D(G,M,I)$.
  Namely, 
  \begin{equation*}
    H^0(\shf{S}) \cong \bigoplus_{\sigma \in D(G,M,I)} \vspan \{\widehat{M_\sigma}\}.
  \end{equation*}
\end{lemma}
We note in passing that the cardinality of $\widehat{M_\sigma}$ is called the ``differential weight'' in \cite{Robinson_relations}.
\begin{proof}
  Lemma \ref{lem:r0_h0} asserts that if $(G,M,I) \in \cat{Ctx}_+$,
  then the space of global sections of $\shf{S}$ is spanned by the elements of $M$.
  Corollary \ref{cor:M_sig_hat} asserts that each element of $M$ identifies exactly one nonempty simplex of $D(G,M,I)$, so the result follows.
\end{proof}

The algebraic structure of the cochain complex of $\shf{S}$ is rather limited,
in that it is \textbf{acyclic}: $H^k(\shf{S}) = 0$ for $k>0$.
Acyclicity suggests that there is a more complete algebraic representation of the data in $(G,M,I)$ than the cochain complex,
a fact which is borne out in Section \ref{sec:dowker_chain_complex}.

\begin{lemma} (standard, for instance, \cite[Thm. II.3.5]{Iversen})
  \label{lem:flabby}
  Suppose that $\shf{F}$ is a sheaf of vector spaces on a topological space and
  that every restriction map from the space of global sections $\shf{F}(U \subseteq X):\shf{F}(X) \to \shf{F}(U)$ is surjective.
  Then $\shf{F}$ is acyclic; the cohomology of $\shf{F}$ satisfies $H^k(\shf{F}) = 0$ for all $k>0$.
\end{lemma}
Such a sheaf is usually called \textbf{flasque} or \textbf{flabby}. 

\begin{lemma}
  \label{lem:r0_hk}
  If $(G,M,I) \in \cat{Ctx}_+$, the Dowker sheaf $\shf{S}= ShvRep(G,M,I)$ is acylic, in that $H^k(\shf{S}) = 0$ for $k>0$.
\end{lemma}
\begin{proof}
  \emph{Every} restriction map of $\shf{S}$---along every inclusion of open sets, not just between two stalks---is surjective.
  This means that $\shf{S}$ is flabby, so the result will follow from Lemmas \ref{lem:cech} and \ref{lem:flabby}.  

  To establish the claim, suppose that $U$ is an open set in the Alexandrov topology of $D=D(G,M,I)$.
  We must show that $\shf{S}(U \subseteq D) : \shf{S}(D) \to \shf{S}(U)$ is surjective.
  By Lemma \ref{lem:r0_h0}, we already know $\shf{S}(D)$ is spanned by the elements of $M$.
  By definition, an open set $U$ is defined as a union of simplices and all their cofaces.
  Explicitly, there is a set of simplices $\{\sigma_j\}_{j \in J} \subseteq D(G,M,I)$ such that
  \begin{equation*}
    U = \bigcup_{j \in J} \{ \tau \in D(G,M,I) : \sigma_j \subseteq \tau \}.
  \end{equation*}
  Observe that each stalk $\shf{S}(\sigma_j) = \vspan \sigma_j'$.
  
  By the same reasoning as in the proof of Lemma \ref{lem:r0_h0}, where one identifies common elements of $M$,
  \begin{equation*}
    \shf{S}(U) \cong \vspan \left(\bigcup_{j \in J} \sigma_j'\right).
  \end{equation*}
  Namely, $\shf{S}(U)$ is spanned by all of the elements of $M$ that are involved in any coface of a simplex in $U$.
  According to Lemma \ref{lem:r0_h0_splitting},
  $\shf{S}(D)$ decomposes into a direct sum, in which each simplex $\sigma_j$ selects a factor corresponding to the simplices involved.
  Since each element of $M$ appears exactly once in $\shf{S}(D)$ by Corollary \ref{cor:M_sig_hat} and restrictions of $\shf{S}$ act by projection,
  it follows that the restriction map is surjective.
\end{proof}
\begin{example}
\label{eg:dowker_r0_hk_example}
Let us continue Example \ref{eg:running} from previous sections.  Recall that the example uses the following incidence matrix,
\begin{equation*}
  I=\begin{pmatrix}
  1& 1& 1& 0& 0& 1\\
  0& 0& 0& 1& 1& 1\\
  0 &1& 0& 0& 1 &1\\
  1 &1& 1& 0 &0 &0\\
\end{pmatrix},
\end{equation*}
in which we label the rows with objects $G = \{a,b,c,d\}$ and the columns with attributes $M=\{0,1,2,3,4,5\}$.

The resulting Dowker sheaf diagram is of the form shown in Figure \ref{fig:dowker_r0_hk_example}.

\begin{figure}[h]
    \centering
    \includegraphics[width=6in]{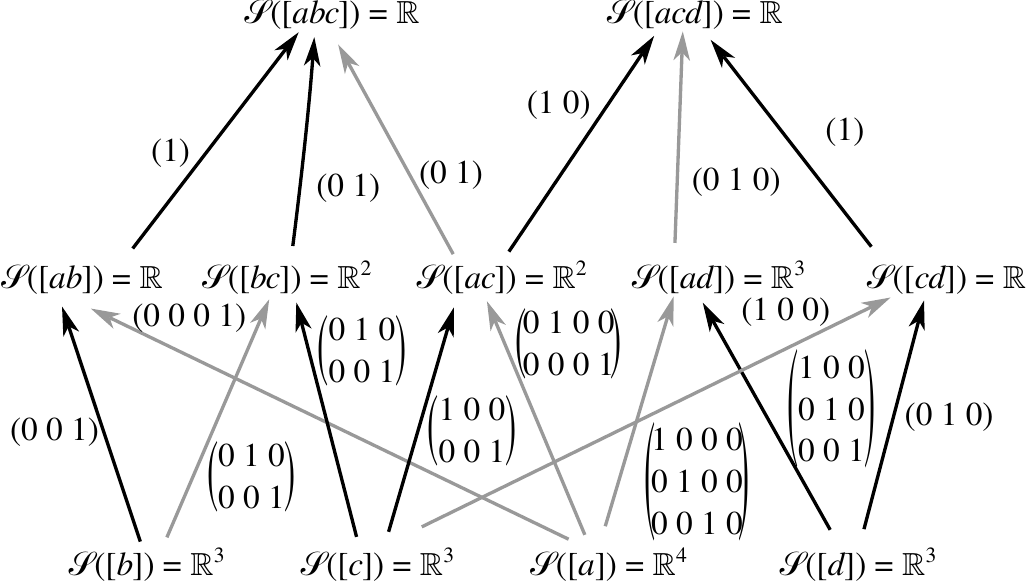}
    \caption{Dowker sheaf representation in Example \ref{eg:dowker_r0_hk_example}; shaded arrows indicate negative signs in the coboundary matrix.}
    \label{fig:dowker_r0_hk_example}
\end{figure}


The cochain complex for the Dowker sheaf representation $\shf{S} = ShvRep(G,M,I)$ is given by a pair of matrices.
In these matrices, the rows and columns are labeled with the basis elements of the stalks of $\shf{S}$.  Concretely, these basis elements correspond to elements of the attribute set $M$ in the context $(G,M,I)$.  The same attribute can appear in multiple stalks of $\shf{S}$, so some columns in the $d^k$ matrices may correspond to the same attribute.  

Given these remarks, we scan left to right across the $0$-dimensional cells Figure \ref{fig:running_cosheaf}, reading off the costalks as Galois pairs in order: 
\begin{equation*}
    b:345,\, c:145,\, a:0125,\, d:012.
\end{equation*}
With this ordering, the the columns of $d^0$ are $\{3,4,5,1,4,5,0,1,2,5,0,1,2\}$, in which the first three columns are associated with the stalk $\shf{S}([b]) = \vspan\{3,4,5\}$, the next three columns are associated with the stalk $\shf{S}([c])=\vspan\{1,4,5\}$, etc.  

On the other hand, the rows are obtained from the $1$-dimensional cells in Figure \ref{fig:running_cosheaf}, namely
\begin{equation*}
ab:5,\, bc:45,\, ac:15,\, ad:012,\, cd:1.
\end{equation*}
The rows of $d^0$ are the same as the columns of $d^1$ are therefore $\{5,4,5,1,5,0,1,2,1\}$.

Given the above setup, the cochain complex coboundary matrices are
\begin{equation*}
  d^0=\left(\begin{array} {ccc|ccc|cccc|ccc} 
0&0&1& &&& 0&0&0&-1& &&\\
\hline
0&-1&0& 0&1&0& &&&& &&\\
0&0&-1& 0&0&1& &&&& &&\\
\hline
&&& 1&0&0& 0&-1&0&0& &&\\
&&& 0&0&1& 0&0&0&-1& &&\\
\hline
&&& &&& -1&0&0&0& 1&0&0\\
&&& &&& 0&-1&0&0& 0&1&0\\
&&& &&& 0&0&-1&0& 0&0&1\\
\hline
&&& -1&0&0& &&&& 0&1&0\\
  \end{array}\right),
\end{equation*}
and
\begin{equation*}
  d^1=\left(\begin{array}{c|cc|cc|ccc|c} 
  1 & 0&1 & 0&-1 & & & &  \\
  \hline
  &  &  & 1 & 0 & 0 & -1 & 0 & 1
  \end{array}\right),
\end{equation*}
where blank blocks correspond to blocks of zeros.
It is easy to verify that $d^1 \circ d^0 = 0$.

After a straightforward, but tedious, calculation of cohomology, we obtain Betti numbers $\dim H^0 = 6$, $\dim H^k = 0$ for $k>0$ in accordance with Proposition \ref{prop:r0_h}.

The generators of $H^0$ are
\begin{equation*}
\begin{pmatrix}
1\\ 0\\ 0\\\hline 0\\ 0\\ 0\\\hline 0\\ 0\\ 0\\ 0\\\hline 0\\ 0\\ 0
\end{pmatrix},
\begin{pmatrix}
0\\ 1\\ 0\\\hline 0\\ 1\\ 0\\\hline 0\\ 0\\ 0\\ 0\\\hline 0\\ 0\\ 0
\end{pmatrix},
\begin{pmatrix}
0\\ 0\\ 1\\\hline 0\\ 0\\ 1\\\hline 0\\ 0\\ 0\\ 1\\\hline 0\\ 0\\ 0
\end{pmatrix},
\begin{pmatrix}
0\\ 0\\ 0\\\hline 0\\ 0\\ 0\\\hline 1\\ 0\\ 0\\ 0\\\hline 1\\ 0\\ 0
\end{pmatrix},
\begin{pmatrix}
0\\ 0\\ 0\\\hline 1\\ 0\\ 0\\\hline 0\\ 1\\ 0\\ 0\\\hline 0\\ 1\\ 0
\end{pmatrix},
\begin{pmatrix}
0\\ 0\\ 0\\\hline 0\\ 0\\ 0\\\hline 0\\ 0\\ 1\\ 0\\\hline 0\\ 0\\ 1
\end{pmatrix}.
\end{equation*}

The generators correspond to columns of our original incidence matrix.
They correspond to
 \begin{enumerate}
 \item  Column 3 (on vertex b),
 \item  Column 4 (on vertices b and c),
 \item  Column 5 (on vertices a, b, and c),
 \item  Column 0 (on vertices a and d),
 \item  Column 1 (on vertices a, c, and d), and
 \item  Column 2 (on vertices a and d).
 \end{enumerate}
\end{example}

\subsection{Representation as a cosheaf of chain complexes}
\label{sec:dowker_chain_complex}

Consider the covariant functor $C^\Delta : \cat{Asc} \to \cat{Kom}$ that constructs the simplicial chain complex from an abstract simplicial complex.
This functor is famous for being the point of departure for simplicial homology theory \cite[Ch. 2]{Hatcher_2002}.

Since the space of local cosections of of $\cshf{R} := CoShvRep(G,M,I)$ are themselves abstract simplicial complexes, we can apply the functor $C^\Delta$ to each space of local cosections.
In this section, we show that this yields a homological invariant for contexts, which is by extension a concept lattice isomorphism invariant as well.  Along the way, we will use one final novel category $\cat{CoShvKom}$, defined below.

\begin{definition}
    \label{def:coshvkom}
    The objects of $\cat{CoShvKom}$ consist of all cosheaves on abstract simplicial complexes of chain complexes.  That is, if $\cshf{C}$ is an object in $\cat{CoShvKom}$, its base space is an abstract simplicial complex, and each space of local cosections is a chain complex.  In short, the global cosections are objects of $\cat{Kom}$.  
    
    The morphisms of $\cat{CoShvKom}$ consist of cosheaf morphisms whose base space maps are simplicial maps, and whose fiber maps are chain maps.
\end{definition}

Although $C^\Delta$ does not preserve colimits in general, it does for the extension maps of the Dowker cosheaf representation.  This follows from a computational lemma.

\begin{lemma}
  Suppose that $W$, $X$, and $Y$ are abstract simplicial complexes such that there are injective simplicial maps
  \begin{equation*}
    \xymatrix{
      X & W \ar[r]^g \ar[l]_f & Y
      }
  \end{equation*}
  such that the simplicial colimit of this diagram is the abstract simplicial complex $Z$.  Then the colimit of the diagram of chain complexes
  \begin{equation*}
    \xymatrix{
      C^\Delta(X) & C^\Delta(W) \ar[r]^{g_\bullet} \ar[l]_{f_\bullet} & C^\Delta(Y)
      }
  \end{equation*}
  (where $f_\bullet$ and $g_\bullet$ are the induced maps) is $C^\Delta(Z)$.
\end{lemma}
\begin{proof}
  The hypothesis on injectivity ensures that $W$ is a closed subspace of both $X$ and $Y$.
  Therefore, $X \sqcup Y$ is an abstract simplicial complex that contains exactly two copies of each simplex of $W$.
  Moreover, the simplicial colimit $Z = (X \sqcup Y)/W$ is simply the adjunction space formed by gluing $X$ and $Y$ together along $W$.
  It therefore contains exactly one copy of each simplex of $W$.

  Because $C^\Delta$ is left exact, the injectivity of $f$ and $g$ also implies that the induced maps $f_\bullet$ and $g_\bullet$ are also injective.
  The reader will then immediately recognize the Mayer-Vietoris short exact sequence applies\footnote{But not the Mayer-Vietoris long exact sequence for homology, though!}, namely,
  \begin{equation*}
    \xymatrix{
      0 \ar[r] & C^\Delta(W) \ar[rr]^-{(f_\bullet,-g_\bullet)} && C^\Delta(X) \oplus C^\Delta(Y) \ar[r] & C^\Delta(Z) \ar[r] & 0
      }
  \end{equation*}
  Since we are working with chain complexes of vector spaces, this short exact sequence always splits.  That is,
  \begin{equation*}
    C^\Delta(X) \oplus C^\Delta(Y) \cong C^\Delta(W) \oplus C^\Delta(Z).
  \end{equation*}
  Notice that the left side contains basis elements corresponding to two copies of each simplex in $W$.
  Therefore, $C^\Delta(Z)$ contains basis elements for \emph{exactly one copy} of each simplex of $W$, since the other copies are all in $C^\Delta(W)$.
  We therefore conclude that the colimit of the chain complex diagram is $C^\Delta(Z)$.
\end{proof}

\begin{corollary}
  Applying $C^\Delta$ to each space of local cosections of $\cshf{R} = CoShvRep(G,M,I)$ yields a well-defined cosheaf
  $\cshf{R}_\bullet := C^\Delta \circ CoShvRep(G,M,I)$ of chain complexes on the Dowker complex $D(G,M,I)$.
\end{corollary}

From the previous sections, this means we can simply compose functors to transform a concept lattice into a context and then into this new cosheaf.

\begin{theorem}
  \label{thm:concept_iso_cosheaf_of}
  The Dowker cosheaf $\cshf{R}_\bullet = C^\Delta \circ CoShvRep(G,M,I)$ of chain complexes is a concept lattice isomorphism invariant.
  Namely, it is a functor $\cat{ConLat} \to \cat{CoShvKom}$
\end{theorem}

The remainder of this section explores properties of this new cosheaf of chain complexes.

The cosheaf chain complex for $\cshf{R}_\bullet = C^\Delta \circ CoShvRep(G,M,I)$ is a double complex in $\cat{Vec}$, that is, it is a chain complex of chain complexes.
It has the structure of a grid of vector spaces
\begin{equation*}
  \xymatrix{
                  & \vdots \ar[d] & \vdots \ar[d] & \vdots \ar[d] \\
    \dotsb \ar[r] & \bigoplus_{\sigma \in \left(D(G,M,I)\right)^2} \left(\cshf{R}(\sigma)\right)_2 \ar[r] \ar[d] & \bigoplus_{\sigma \in \left(D(G,M,I)\right)^2} \left(\cshf{R}(\sigma)\right)_1 \ar[r] \ar[d] & \bigoplus_{\sigma \in \left(D(G,M,I)\right)^2} \left(\cshf{R}(\sigma)\right)_0 \ar[d] \\
    \dotsb \ar[r] & \bigoplus_{\sigma \in \left(D(G,M,I)\right)^1} \left(\cshf{R}(\sigma)\right)_2 \ar[r] \ar[d] & \bigoplus_{\sigma \in \left(D(G,M,I)\right)^1} \left(\cshf{R}(\sigma)\right)_1 \ar[r] \ar[d] & \bigoplus_{\sigma \in \left(D(G,M,I)\right)^1} \left(\cshf{R}(\sigma)\right)_0 \ar[d] \\
    \dotsb \ar[r] & \bigoplus_{\sigma \in \left(D(G,M,I)\right)^0} \left(\cshf{R}(\sigma)\right)_2 \ar[r] & \bigoplus_{\sigma \in \left(D(G,M,I)\right)^0} \left(\cshf{R}(\sigma)\right)_1 \ar[r]  & \bigoplus_{\sigma \in \left(D(G,M,I)\right)^0} \left(\cshf{R}(\sigma)\right)_0 \\
    }
\end{equation*}
in which the rows are indexed by the dimension of the simplices of $D(G,M,I)$,
and the columns are indexed by the chain complexes within each costalk of $\cshf{R}_\bullet$,
which themselves are indexed by the dimensions of $D(M,G,I^T)$.

\begin{corollary}
  The last column of the above diagram is isomorphic to the cosheaf chain complex of $\cshf{R}^0$,
  and hence (according to duality and Proposition \ref{prop:r0_h}) has nontrivial homology only in index $0$, where it takes dimension $\#M$.
\end{corollary}

Because we work over vector spaces, vector space duality ensures that we can obtain a sheaf of cochain complexes on the (same) Dowker complex as well.  
In short, the structure of the diagram remains the same, but all the arrows reverse due to matrix transposition.
This explains why $H^\bullet(\shf{S})$ is lossy.
It only considers the last column of the grid diagram.

\begin{lemma}
\label{lem:chain_homology}
  Each column of the above diagram is a chain complex which has nontrivial homology only in index $0$.
\end{lemma}
\begin{proof}
  Dualize the vertical maps using vector space duality.
  These are each surjective by essentially the same reasoning in as in Lemma \ref{lem:r0_hk}.
  That is, for $\sigma \in D(G,M,I)$,
  each generator of each space $\left(\cshf{R}(\sigma)\right)_k$ corresponds to a simplex of $D(M,G,I^T)$ of dimension $k$.

  More explicitly, surjectivity is a consequence of the fact that the costalks of $\cshf{R}$ are individual simplices of the dual Dowker complex, since for a context $(G,M,I)$, the costalks are precisely the concepts of $\setsys{B}_M$.
As a result, each costalk of $\cshf{R}_\bullet$ is an acyclic chain complex; it can only have nontrivial homology in index $0$.
  If $k>0$, then each such simplex is also a simplex of $\left(\cshf{R}(\tau)\right)_k$ for every face $\tau \subseteq \sigma$.  
\end{proof}

The reader is cautioned that row-wise homology may differ from column-wise homology!

\begin{proposition}
  The zeroth cosheaf homology of the cosheaf of chain complexes 
  \begin{equation*}
  \cshf{R}_\bullet = C^\Delta \circ CoShvRep(G,M,I)
  \end{equation*}
  is the simplicial chain complex of the Dowker complex $D(M,G,I^T)$.  
\end{proposition}
\begin{proof}
  Notice that the object of interest is the last row of column-wise homology of the grid diagram.
  Since $C^\Delta \circ CoShvRep(G,M,I)$ preserves colimits,
  and the space of global cosections (i.e.\ the colimit) of $CoShvRep(G,M,I)$ is $D(M,G,I^T)$, the result follows easily from a little bookkeeping.
\end{proof}

\begin{corollary}
\label{cor:cosheaf_of_generalizes}
  The homology of the zeroth cosheaf homology of the cosheaf of chain complexes $\cshf{R}_\bullet = C^\Delta \circ CoShvRep(G,M,I)$ is the simplicial homology of the Dowker complex $D(M,G,I^T)$.
\end{corollary}



\section*{Acknowledgements}

Information release PNNL-SA-210092.

Pacific Northwest National Laboratory is a multiprogram national laboratory operated for the US Department of Energy (DOE) by Battelle Memorial Institute under Contract No. DE-AC05-76RL01830. 

Robinson was partially supported by the Defense Advanced Research Projects Agency (DARPA) SafeDocs program under contract HR001119C0072.
Any opinions, findings and conclusions or recommendations expressed in this material are those of the authors and do not necessarily reflect the views of DARPA.

Thanks to Emilie Purvine at PNNL for an excellent peer review.

\bibliographystyle{plain}

\bibliography{./reference}

\begin{thebibliography}{10}

\bibitem{ayzenberg2019topology}
Anton Ayzenberg.
\newblock Topology of nerves and formal concepts, {\tt arxiv:1911.05491}, 2019.

\bibitem{Barmak2011}
Jonathan~A. Barmak.
\newblock {\em Algebraic Topology of Finite Topological Spaces and
  Applications}.
\newblock Springer Berlin Heidelberg, 2011.

\bibitem{BeAAbR18}
Austin~R Benson, Rediet Abebe, Mciahel~T Schaub, A~Jadbabaie, and Jon
  Kleinberg.
\newblock Simplicial closure and higher-order link prediction.
\newblock {\em Proc. National Academic of Science}, 115:48, 2018.
\newblock https://doi.org/10.1073/pnas.1800683115.

\bibitem{BeC89}
C~Berge.
\newblock {\em Hypergraphs: Combinatorics of Finite Sets}.
\newblock Elsevier, 1989.

\bibitem{birkhoff}
Garrett Birkhoff.
\newblock {\em Lattice Theory}.
\newblock American Mathematical Society Colloquium Publications. American
  Mathematical Society, Providence, RI, 2nd edition, 1948.

\bibitem{Bloch2017}
Isabelle Bloch.
\newblock Morphological links between formal concepts and hypergraphs.
\newblock In {\em Lecture Notes in Computer Science}, pages 16--27. Springer
  International Publishing, 2017.

\bibitem{Bradley2024}
Tai-Danae Bradley, Juan~Luis Gastaldi, and John Terilla.
\newblock The structure of meaning in language: Parallel narratives in linear
  algebra and category theory.
\newblock {\em Notices of the American Mathematical Society}, 71(02):1,
  February 2024.

\bibitem{Bredon}
Glen Bredon.
\newblock {\em Sheaf theory}.
\newblock Springer, 1997.

\bibitem{BrA13}
Alain Bretto.
\newblock {\em Hypergraph Theory}.
\newblock Springer-Verlag, Heidelberg, 2013.

\bibitem{brun2019sparse}
Morten Brun and Nello Blaser.
\newblock Sparse {Dowker} nerves.
\newblock {\em Journal of Applied and Computational Topology}, 3(1-2):1--28,
  2019.

\bibitem{Cattaneo2016}
Gianpiero Cattaneo, Giampiero Chiaselotti, Davide Ciucci, and Tommaso Gentile.
\newblock On the connection of hypergraph theory with formal concept analysis
  and rough set theory.
\newblock {\em Information Sciences}, 330:342--357, February 2016.

\bibitem{Curry}
J.~Curry.
\newblock {\em Sheaves, Cosheaves and Applications, {\tt arXiv:1303.3255}}.
\newblock PhD thesis, University of {Pennsylvania}, 2013.

\bibitem{davey_priestley_2002}
B.~A. Davey and H.~A. Priestley.
\newblock {\em Introduction to Lattices and Order}.
\newblock Cambridge University Press, 2 edition, 2002.

\bibitem{Drfler1980}
W.~D{\"o}rfler and D.~A. Waller.
\newblock A category-theoretical approach to hypergraphs.
\newblock {\em Archiv der Mathematik}, 34(1):185–192, December 1980.

\bibitem{Dowker1952}
C.~H. Dowker.
\newblock Homology groups of relations.
\newblock {\em The Annals of Mathematics}, 56(1):84, July 1952.

\bibitem{FoBSpD18}
Brendan Fong and David~I Spivak.
\newblock {\em Seven Sketches in Compositionality: An Invitation to Applied
  Category Theory}.
\newblock Cambridge University Press, 2018.

\bibitem{Freund2015}
Anton Freund, Moreno Andreatta, and Jean-Louis Giavitto.
\newblock Lattice-based and topological representations of binary relations
  with an application to music.
\newblock {\em Annals of Mathematics and Artificial Intelligence},
  73(3-4):311--334, January 2015.

\bibitem{GaBWiR99}
Bernhard Ganter and Rudolf Wille.
\newblock {\em Formal Concept Analysis}.
\newblock Springer-Verlag, 1999.

\bibitem{purvine_Hgraph_homology}
Ellen Gasparovic, Emilie Purvine, Radmila Sazdanovic, Bei Wang, Yusu Wang, and
  Lori Ziegelmeier.
\newblock A survey of simplicial, relative, and chain complex homology theories
  for hypergraphs, {\tt arxiv:2409.18310}, 2024.

\bibitem{grilliette2024simplificationincidenceincidencefocused}
Will Grilliette.
\newblock Simplification \& incidence: How an incidence-focused perspective
  patches category-theoretic problems in graph theory, {\tt arxiv:2403.13165},
  2024.

\bibitem{Grilliette2018IncidenceHT}
Will Grilliette and Lucas~J. Rusnak.
\newblock Incidence hypergraphs: the categorical inconsistency of set-systems
  and a characterization of quiver exponentials.
\newblock {\em Journal of Algebraic Combinatorics}, 58:1--36, 2018.

\bibitem{Hatcher_2002}
A.~Hatcher.
\newblock {\em Algebraic Topology}.
\newblock Cambridge University Press, 2002.

\bibitem{Hubbard_2006}
John~H. Hubbard.
\newblock {\em Teichm\"uller Theory, volume 1.}
\newblock Matrix Editions, 2006.

\bibitem{Iversen}
B.~Iversen.
\newblock {\em Cohomology of Sheaves}.
\newblock Aarhus universitet, Matematisk institut, 1984.

\bibitem{Krtzsch2005}
Markus Kr\"{o}tzsch, Pascal Hitzler, and Guo-Qiang Zhang.
\newblock {\em Morphisms in Context}, page 223–237.
\newblock Springer Berlin Heidelberg, 2005.

\bibitem{nlab:lat}
{nLab authors}.
\newblock {{L}}at.
\newblock \url{https://ncatlab.org/nlab/show/Lat}, June 2023.
\newblock \href{https://ncatlab.org/nlab/revision/Lat/1}{Revision 1}.

\bibitem{priestley_1970}
H.~A. Priestley.
\newblock Representation of distributive lattices by means of ordered stone
  spaces.
\newblock {\em Bulletin of the London Mathematical Society}, 2(2):186--190,
  1970.

\bibitem{Robinson_relations}
Michael Robinson.
\newblock Cosheaf representations of relations and {Dowker} complexes.
\newblock {\em Journal of Applied and Computational Topology}, 2021.

\bibitem{spanier}
Edwin Spanier.
\newblock {\em Algebraic topology}.
\newblock Springer, 1966.

\bibitem{StR97}
Richard~P Stanley.
\newblock {\em Enumerative Combinatorics: vol. 1}.
\newblock Cambridge UP, Cambridge, 1997.

\bibitem{Stell2014}
John~G. Stell.
\newblock Formal concept analysis over graphs and hypergraphs.
\newblock In {\em Lecture Notes in Computer Science}, pages 165--179. Springer
  International Publishing, 2014.

\bibitem{weibel}
Charles Weibel.
\newblock {\em An Introduction to Homological Algebra}.
\newblock Cambridge, 1994.

\end{thebibliography}

\end{document}